\documentclass[12pt]{amsart}
\usepackage[margin=1in]{geometry}
\usepackage{latexsym}
\usepackage{float}
\usepackage{amssymb}
\usepackage[all]{xy}
\usepackage{epsfig}
\usepackage{color}
\usepackage{graphicx}
\usepackage{hyperref}

\newtheorem{Thm}[equation]{Theorem}
\newtheorem{Lem}[equation]{Lemma}
\newtheorem{Cor}[equation]{Corollary}
\newtheorem{Con}[equation]{Conjecture}

\theoremstyle{remark}

\theoremstyle{definition}

\numberwithin{equation}{section}

\begin{document}

\title{Uniqueness of reduced alternating rational 3-tangle diagrams}
\author{ Bo-hyun Kwon}
\date{April 2015}

\begin{abstract}
Tangles were introduced by J. Conway.  In 1970, he proved that every rational 2-tangle defines a rational number and two rational 2-tangles are isotopic if and only if they have the same rational number.
So, from Conway's result we have a perfect classification for rational $2$-tangles.
However, there is no similar theorem to classify rational $3$-tangles.

 In this paper, we introduce an invariant of rational $n$-tangles which is obtained from the Kauffman bracket. It forms a vector with Laurent polynomial entries. We prove that the invariant classifies  the rational $2$-tangles and the reduced alternating rational $3$-tangles.  We conjecture that it classifies the rational $3$-tangles as well. 
\end{abstract}
\maketitle
\section{Introduction}
A $n$-$tangle$ is the disjoint union of $n$ properly embedded arcs in the unit 3-ball. 
  A $rational$ $n$-$tangle$ is a $n$-tangle $\alpha_1\cup\alpha_2\cup\cdot\cdot\cdot \cup\alpha_n$ in a 3-ball $B^3$
 such that there exists a homeomorphism of pairs $\Phi: (B^3,\alpha_1\cup\alpha_2\cup\cdot\cdot\cdot \cup\alpha_3)\longrightarrow
(D^2\times I,\{p_1,p_2,..., p_n\}\times I)$). Then two rational $n$-tangles, $T,T'$, in $B^3$ are $\emph{isotopic}$, denoted by $T\approx T'$, if there is an orientation-preserving self-homeomorphism
$h: (B^3, T)\rightarrow (B^3,T')$ that is the identity map on the boundary.\\

 H. Cabrera-Ibarra~\cite{3} found a pair of invariants which is defined for all rational 3-tangles. Each invariant is a $3\times 3$ matrix with complex number entries.
  The pair of invariants classifies the elements of six special sets of rational 3-tangles each of which contains the braid 3-tangles. However, it does not cover the collection of all alternating rational 3-tangles.\\
  
  Also, Emert and Ernst~\cite{5} classified  the collection of  $\textit{essential}$ alternating rational $n$-tangles which is a set of alternating rational $n$-tangles satisfying a certain condition. However, it also does not cover the collection of all alternating rational 3-tangles.\\
  
  In this paper, I would like to classify the collection of all alternating rational $3$-tangles. With a similar argument, it could be possible to classify the collection of all alternating rational $n$-tangles.\\

Let $S^2$ be a sphere smoothly embedded in $S^3$ and let $K$ be a link transverse to $S^2$. The complement in $S^3$ of $S^2$ consists of two open balls, $B_1$ and $B_2$. We assume that $S^2$ is $xz$-plane $\cup~\{\infty\}$.  Now, consider the projection of $K$ onto the flat $xy$-plane.
Then, the projection onto the $xy$-plane of $S^2$ is the $x$-axis and $B_1$ projects to the upper half plane and $B_2$ projects to the lower half plane.
The projection gives us a {\emph{link diagram}}, where we make note of over and undercrossings.
The diagram of the link $K$  is called a $\emph{plat on 2n-strings}$, denoted by $p_{2n}(w)$,  if it is the union of a $2n$-braid $w$ and $2n$ unlinked and unknotted arcs which connect pairs of consecutive strings of the braid at the top and at the bottom endpoints and $S^2$ meets the top of the $2n$-braid. (See  the first and second diagrams of Figure~\ref{p1}.) Any link $K$ in $S^3$ admits a plat presentation. i.e., $K$ is ambient isotopic to a plat (\cite{2}, Theorem 5.1). The bridge (plat) number $b(K)$ of $K$ is the smallest possible number $n$ such that there exists a plat presentation  of $K$ on $2n$ strings. We say that $K$ is $n$-bridge link if the bridge number of $K$ is $n$. We remark that the braid group $\mathbb{B}_{2n}$ is generated by $\sigma_1,\sigma_2,\cdot\cdot\cdot\sigma_{2n-1}$ which are twisting of two adjacent strings. For example, $w=\sigma_2^{-1}\sigma_4^{-1}\sigma_3\sigma_1^3\sigma_5^2\sigma_4^{-1}\sigma_2^{-1}$ is the word for the
6 braid of the first diagram of Figure~\ref{p1}.
\\

\begin{figure}[htb]
\includegraphics[scale=.40]{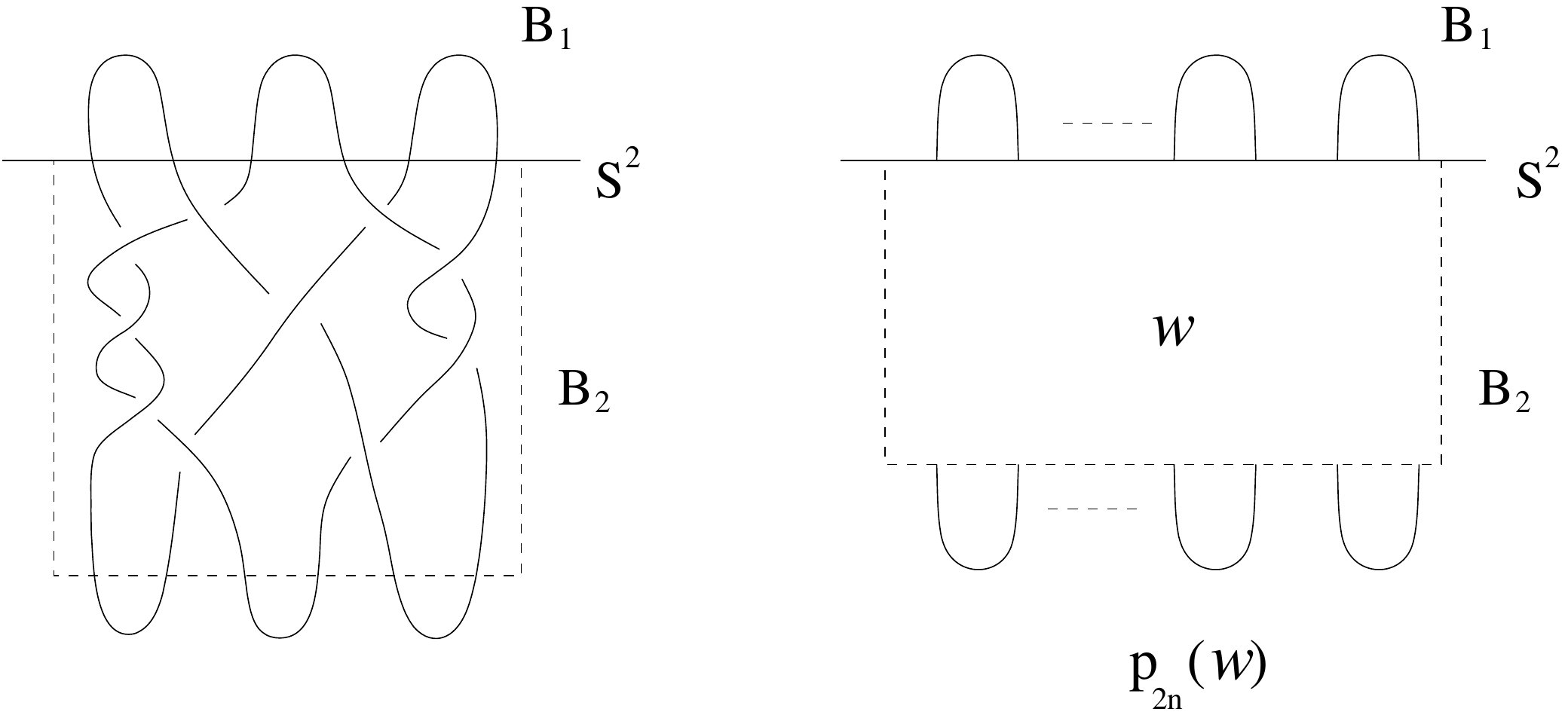}
\caption{}
\label{p1}
\end{figure}

 Then we say that a plat presentation  is $\emph{standard}$ if the $2n$-braid $w$ of $p_{2n}(w)$  involves only $\sigma_2,\sigma_3,\cdot\cdot\cdot,\sigma_{2n-1}$.\\
 
 We remark that $K\cap B_2$ is a rational tangle.\\
 
Now, we define a $\emph{2n-plat presentaion}$  for rational $n$-tangles $K\cap B_2$   in $B_2$, denoted by $q_{2n}(w)$, if it is  the union of a $2n$-braid $w$ and $n$ unlinked and unknotted arcs which connect pairs of strings of the braid  at the bottom endpoints with the same pattern as in a plat presentation  for a link  and the projection of $\partial B_2$ onto the flat $xy$-plane meets the top of the $2n$-braid.\\

We note that $q_{2n}(w)$ is  a rational $n$-tangle in $B_2$ since we can obtain a trivial rational $n$-tangle from the rational $n$-tangle by a sequence of half Dehn twists which are automorphisms of $B^3$ that preserve the six punctures.\\

We also say that $\overline{q_{2n}(w)}$ (= $p_{2n}(w)$) is the $\emph{plat closure}$ of $q_{2n}(w)$ if it is the union of $q_{2n}(w)$ and $n$ unlinked and unknotted arcs in $B_1$ which connect pairs of consecutive strings of the braid at the top endpoints. \\

The tangle diagrams with the circles in Figure~\ref{p2} give the  diagrams of trivial rational $2,3$-tangles as in~\cite{1},~\cite{3},~\cite{6},~\cite{8}. The right sides of each pair of diagrams show the trivial rational $2,3$-tangles in $B_2$.\\

 A tangle diagram $TD$ (or $2n$-plat presentation) is  $\emph{reduced}$ alternating if $TD$ is alternating and $TD$ does not have a self-crossing which can be removed by a Type I Reidemeister move.\\

We note that $q_4(w)$ is alternating if and only if $\overline{q_4(w)}$ is alternating, possibly not reduced alternating.\\

\begin{figure}[htb]
\includegraphics[scale=.32]{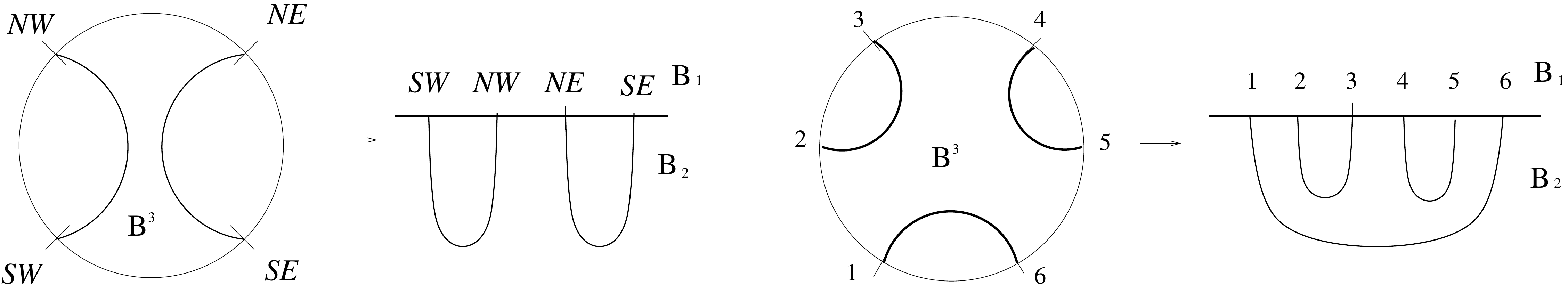}
\caption{}
\label{p2}
\end{figure}

In section~\ref{s2}, we introduce the Kauffman bracket of a rational tangle diagram and discuss how to calculate it. Also, we define a vector from the Kauffman bracket of a rational $2$-tangle diagram.\\

Then, we will prove that the vector from the Kauffman bracket of a rational $2$-tangle diagram is an invariant which can classify the rational 2-tangles in section~\ref{s3}.\\

Finally, we will show that the vector from the Kauffman bracket of a rational $3$-tangle diagram is an invariant of rational $3$-tangles and especially it classifies the reduced alternating rational 3-tangles in section~\ref{s4}.

\section{The Kauffman bracket and its calculation}\label{s2}

Let $\Lambda=\mathbb{Z}[a,a^{-1}]$ and $L$ be a link diagram. I want to emphasize here that $K,L$ and $T$ stand for a diagram and $\mathbb{K}$, $\mathbb{L}$ and $\mathbb{T}$ stand for a knot or link for convenience.  \\

We recall that the Kauffman bracket $<L>\in \Lambda$ of  $L$ is obtained from the three axioms
\begin{figure}[htb]
\begin{center}
\includegraphics[scale=.4]{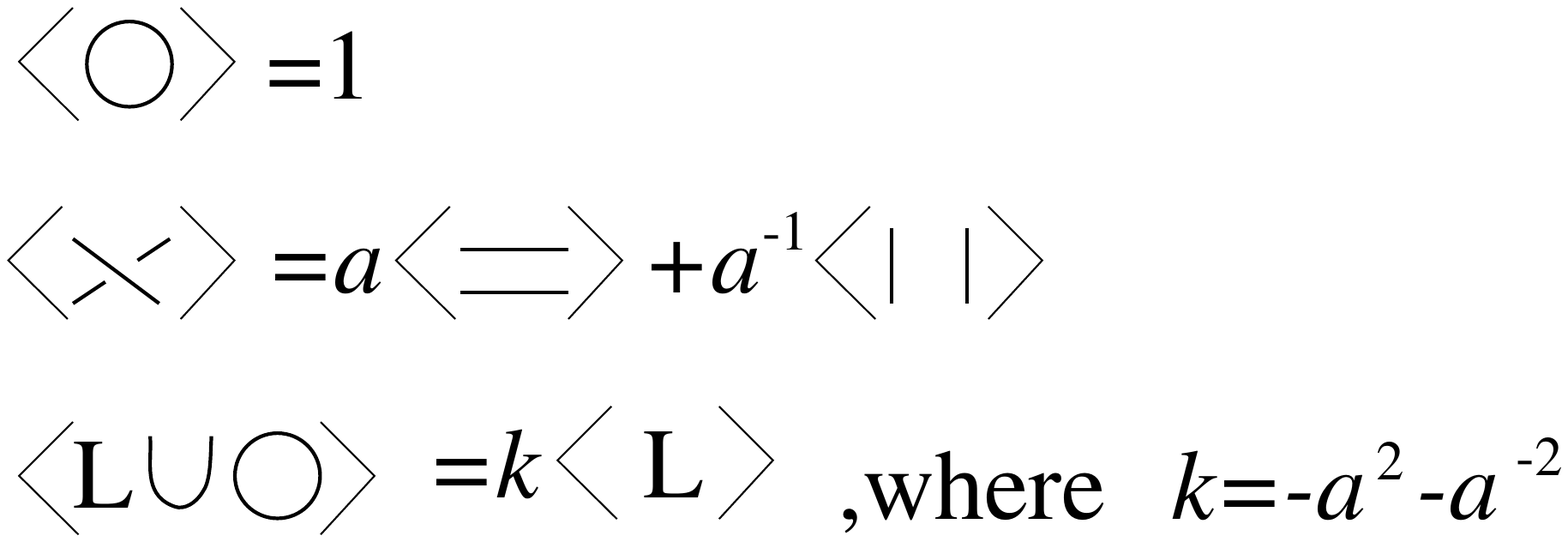}
\end{center}
\vskip -270pt
\end{figure}

The symbol $< ~>$ indicates that the changes are made to the diagram locally, while the rest of the diagram is fixed.\\

 The Kauffman polynomial $X_L(a)\in \Lambda$ is defined by 
\begin{center}
$X_L(a)=(-a^{-3})^{w(\overrightarrow{L})}<L>$,
\end{center}
where the writhe $w(\overrightarrow{L})\in\mathbb{Z}$ is obtained by assigning an orientation to $L$, and taking a sum over all crossings of $L$ of their indices $e$, which are given by the following rule
 
 \begin{figure}[htb]
 \begin{center}
 \includegraphics[scale=.5]{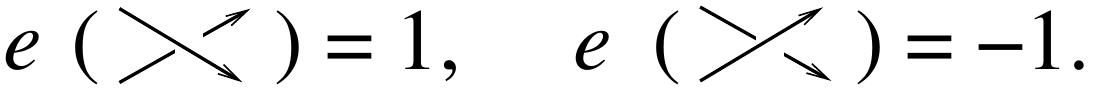}
 \end{center}
 \end{figure}

Then we have the following theorem.

\begin{Thm}(\cite{8})\label{T1}
If $\mathbb{K}$ is a 2-bridge  link, then there exists a word $w$ in $\mathbb{B}_4$ so that the  plat presentation $p_4(w)$ is reduced alternating and standard and represents a link isotopic to $\mathbb{K}$.
\end{Thm}

Since $p_4(w)$ is standard, we consider $\mathbb{B}_3$ instead of $\mathbb{B}_4$.  Then, let $\sigma_1$ and $\sigma_2$ be the two generators of $\mathbb{B}_3$.
I want to emphasize here that we are changing  from $\sigma_2$ and $\sigma_3$ to $\sigma_1$ and $\sigma_2$.
\\

 \begin{figure}[htb]
 \includegraphics[scale=.6]{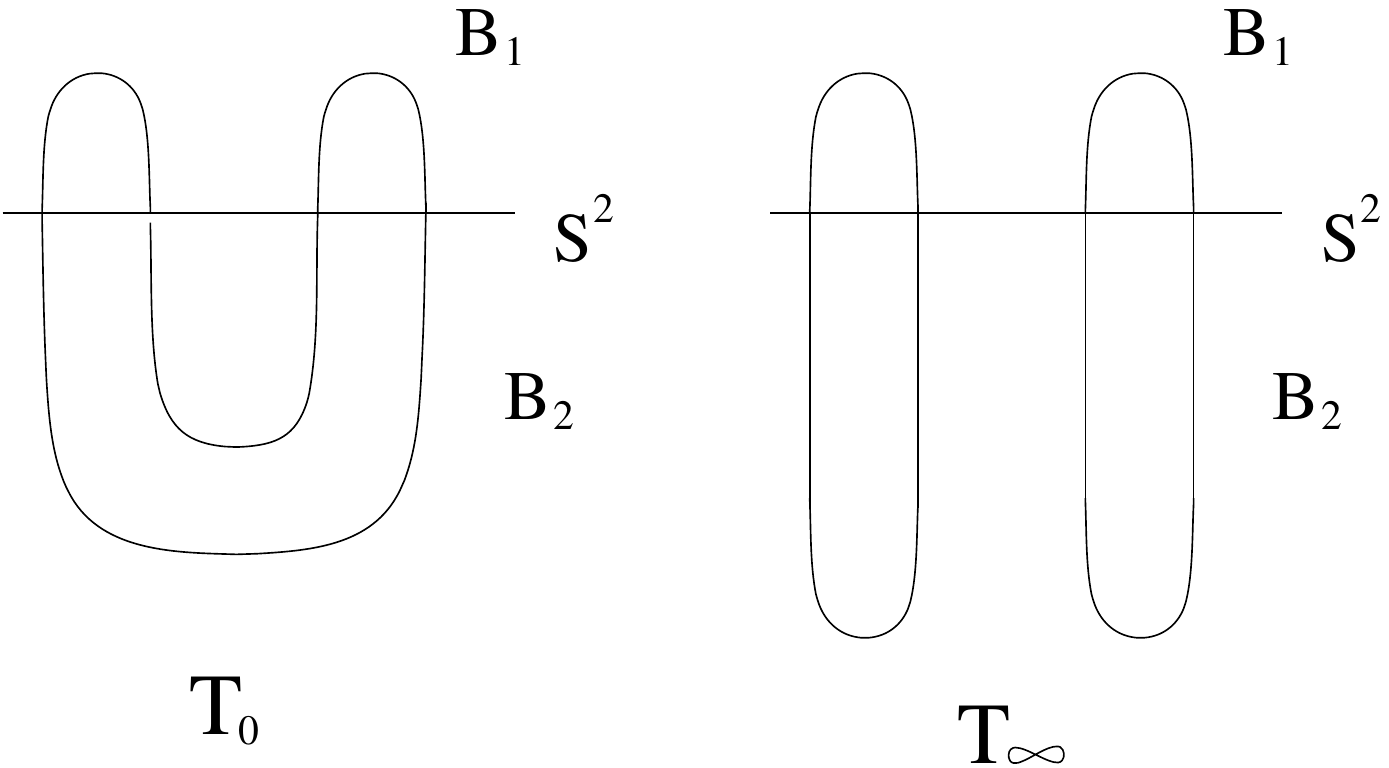}
 \caption{}
 \label{p3}
 \end{figure}

 Goldman and  Kauffman [5] define the $\emph{bracket polynomial}$ of the two tangle diagram $T$  as 
 $<T>=f(a)<T_0>+g(a)<T_{\infty}>$, where the coefficients $f(a)$ and $g(a)$ are Laurent polynomials in $a$ and $a^{-1}$ that are obtained by starting with $T$ and using the three axioms repeatedly until only the two trivial tangle diagrams $T_0$ and $T_{\infty}$ in the expression given for $T$ are left.  We note that $f(a)$ and $g(a)$ are invariant under regular isotopy of $T$, where regular isotopy is the equivalence relation of link diagrams that is generated by using the 2nd and 3rd Reidemeister moves only.
  So, we define the coefficients vector $(f(a),g(a))$ which is a regular invariant of the rational $3$-tangles.\\
 
 Let $\mathcal{A}=<T_0>$ and $\mathcal{B}=<T_{\infty}>$. So, $<T>=f(a)\mathcal{A}+g(a)\mathcal{B}$.\\

We assume that $\mathbb{T}$ is a reduced alternating rational $2$-tangle. Then we have $w=\sigma_1^{\epsilon_1}\sigma_2^{-\epsilon_2}\cdot\cdot\cdot\sigma_1^{\epsilon_{2k-1}}$
for some positive (negative) integers $\epsilon_i$ for $2\leq i\leq 2k-1$ and non-negative (non-positive) integer $\epsilon_1$ for $\mathbb{T}$. 
We note that $w$ needs to end at $\sigma_1^{\pm 1}$. If  $w=\sigma_1^{\epsilon_1}\sigma_2^{-\epsilon_2}\cdot\cdot\cdot\sigma_1^{\epsilon_{2k-1}}\sigma_2^{\epsilon_{2k}}$  for some positive (negative) integer $\epsilon_{2k}$ then $q_4(w)$ is not a reduced alternating tangle diagram. i.e., the tangle diagram has a self crossing which can be removed by the first Reidemeister move.
\\

Suppose that $T_1=q_4(w)$, where $w=\sigma_1^{\epsilon_1}\sigma_2^{-\epsilon_2}\cdot\cdot\cdot\sigma_1^{\epsilon_{2n-1}}$ for some positive (negative) integers $\epsilon_i$ ($2\leq i\leq 2n-1$) and non-negative (non-positive) integer $\epsilon_1$.\\

  Let $A_1^{\pm 1}= \left( \begin{array}{cl}
-a^{\mp 3} & a^{\mp 1} \\
0 & a^{\pm 1} \\
 \end{array} \right)$ \hskip 10pt and \hskip 10pt $A_2^{\pm 1}=\left( \begin{array}{lc}
a^{\pm 1} & 0 \\
a^{\mp 1} & -a^{\mp 3} \\
 \end{array} \right)$.
 
Also, let $A=A_1^{\epsilon_1}A_2^{-\epsilon_2}\cdot\cdot\cdot
 A_1^{\epsilon_{2n-1}}$.\\

Then we can calculate the two coefficients $f(a)$ and $g(a)$ of $\mathcal{A}$ and $\mathcal{B}$ of $T_1$ as follows.\\

\begin{Thm}(Eliahou-Kauffman-Thistlethwaite~\cite{4})\label{T2} Suppose that $T_1$ (=$q_4(u)$) is a plat presentation of a rational $2$-tangle $\mathbb{T}_1$  which is alternating and standard so that
$u=\sigma_1^{\epsilon_1}\sigma_2^{-\epsilon_2}\cdot\cdot\cdot\sigma_1^{\epsilon_{2n-1}}$ for some positive (negative) integers $\epsilon_i$ ($1\leq i\leq 2n-1$) and non-negative (non-positive) integer $\epsilon_1$. Then,
\\

 $<T_1>=f(a)\mathcal{A}+g(a)\mathcal{B}$, where $f(a)$ and $g(a)$ are given by \\

 $(f(a),g(a))^{t} = A \left( \begin{array}{c}
0 \\
 1 \\
 \end{array} \right)$, i.e., the second column of $A$, where $A=A_1^{\epsilon_1}A_2^{-\epsilon_2}\cdot\cdot\cdot A_1^{\epsilon_{2n-1}}$.
 
 \end{Thm}

Then we  show the following lemma which helps us to show Theorem~\ref{T5}.\\

\begin{Lem}\label{T3}
Let $A=A_2^{m}$ for some  non-zero integer $m$.\\

 Then  $A= \left( \begin{array}{cc}
a^m & 0 \\
\displaystyle{{a^{m+2}+(-1)^{m+1}a^{2-3m}\over 1+a^4}} & (-1)^ma^{- 3m} \\
 \end{array} \right).$

\end{Lem}

\begin{proof}
First, assume that $m$ is a positive integer.\\

Let $A= \left( \begin{array}{cc}
b_{11} & b_{12} \\
b_{21} & b_{22} \\
 \end{array} \right)$.\\
 
 I will use induction on $m$ to show this lemma.\\

 We can easily check that if $m=1$ then $b_{11}=a, b_{12}=0, b_{21}=a^{-1}$ and $b_{22}=-a^{-3}$ from $A_2$.\\

Suppose that the claim is true when $m=k$. i.e., $b_{11}=a^{k}$, $b_{12}=0$,  $b_{22}=(-1)^ka^{-3k}$, and\\
 
  $\displaystyle{b_{21}={a^{k+2}+(-1)^{k+1}a^{2-3k}\over 1+a^4}}$.\\

Now, consider $A_2^{k+1}=A_2A_2^{k}= \left( \begin{array}{cc}
b'_{11} & b'_{12} \\
b'_{21} & b'_{22} \\
 \end{array} \right)$.\\

 Then we have $b'_{11}=aa^{k}=a^{k+1}$, $b'_{12}=0$,  $b'_{22}=-a^{-3}(-1)^{k}a^{-3k}=(-1)^{k+1}a^{-3(k+1)}$ and\\
 
  $\displaystyle{b'_{21}=a^{-1}(a^{k})-a^{-3}\left({a^{k+2}+(-1)^{k+1}a^{2-3k}\over 1+a^4}\right)=a^{k-1}-{a^{k-1}+(-1)^{k+1}a^{-1-3k}\over 1+a^4}}$\\
  
  $\displaystyle{={{a^{k-1}(1+a^4)}-(a^{k-1}+(-1)^{k+1}a^{-1-3k})\over 1+a^4}={a^{k+3}+(-1)^{k+2}a^{-1-3k}\over 1+a^4}={a^{(k+2)+1}+(-1)^{(k+1)+1}a^{2-3(k+1)}\over 1+a^4}}.$\\

This completes the proof of the case when $m$ is a positive integer.\\

Similarly, we can show the case when $m$ is a negative integer.

\end{proof}

\section{A new invariant of rational 2-tangles}\label{s3}

For a Laurent polynomial $p(a)$, let $M(p(a))$ be the maximal power of $a$ in $p(a)$ and $m(p(a))$ be the minimal power of $a$ in $p(a)$. We say that a link $K$ is $\textit{connected}$ if no component of $K$ is split. Similarly, we say that  a tangle $T$ is $\textit{connected}$ if no component of $T$ is split.\\

K. Murasugi~\cite{11} proved the following theorem and it helps to prove Theorem~\ref{T5}.

\begin{Thm}[\cite{11}]\label{T4}
Suppose that $K$ is a connected reduced alternating projection of an alternating link $\mathbb{K}$.
Then $(M(X_K)-m(X_K))/4$ is the  number of crossings of $K$.

\end{Thm}

We recall that there is a perfect classification of rational 2-tangles which define a rational number.
The following theorem tells us  the coefficients vector of rational $2$-tangle diagrams classifies the corresponding rational 2-tangles. However, it is not good as the classical invariant. 
\begin{Thm}\label{T5}
Suppose that $T$ is a tangle diagram of a rational $2$-tangle $\mathbb{T}$ so that $<T>=f(a)<T_0>+g(a)<T_{\infty}>$.\\
 
 Then, $\mathbb{T}\approx \mathbb{T}_{\infty}$ if and only if $(f(a),g(a))=(-a^{-3})^k(0,1)$ for some integer $k$, where $\mathbb{T}_{\infty}$ is the tangle with the diagram $T_{\infty}$.
\end{Thm}

 \begin{proof}
 
 First, we suppose that $\mathbb{T}\approx \mathbb{T}_{\infty}$.\\
 
 Since $\mathbb{T}\approx \mathbb{T}_\infty$, we get $T_\infty$ from $T$ after applying a sequence of a finite number of the three Reidemeister moves.
  We note that $f(a)$ and $g(a)$ are invariant under regular isotopy of $T$.\\

Now, consider the first Reidemeister moves as in Figure~\ref{p4}.\\

\begin{figure}[htb]

\includegraphics[scale=.5]{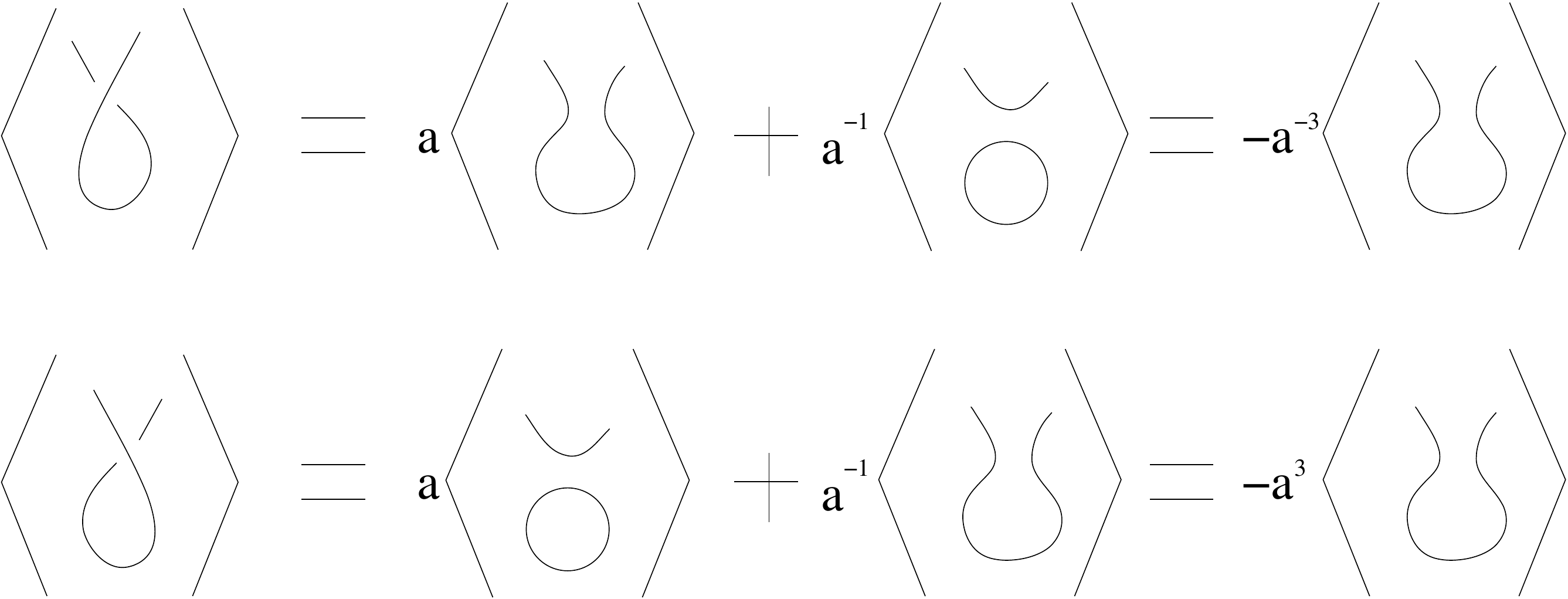}
\caption{}
\label{p4}
\end{figure}

This implies that $(f(a),g(a))=(-a^{-3})^k(0,1)$ for some integer $k$ since $<T_{\infty}>=0 \cdot <T_0>+1 \cdot <T_{\infty}>$.\\

In order to show the opposite direction, assume that there is a non-trivial reduced alternating projection $T$ so that $<T>=(-a^{-3})^k<T_\infty>$.\\

Let $q_4(w)$ be the plat presentation for $T$ which is standard and reduced alternating.\\

Then, either $w=\sigma_1^{\epsilon_1}$ for a non-zero integer $\epsilon_1$ or $w=\sigma_2^{-\epsilon_0}\sigma_1^{\epsilon_1}\sigma_2^{-\epsilon_2}\sigma_1^{\epsilon_3}\cdot\cdot\cdot\sigma_1^{\epsilon_{2n-1}}$ for some positive (negative) integers $\epsilon_i$ ($1\leq i\leq 2n-1$) and non-negative (non-positive) integer $\epsilon_0$.\\

If $w=\sigma_1^{\pm 1}$, then we see that $(f(a),g(a))=(a^{\mp 1},a^{\pm 1})\neq (-a^{-3})^k(0,1)$ for any $k$.\\

If   $w=\sigma_1^{\epsilon_1}$ for  $|\epsilon_1|\geq 2$, then we have $p_4(w)$ which represents a reduced alternating link $\mathbb{K}$ having a diagram $K=\overline{T}$. So, we note that the Kauffman polynomial of $K$ is not trivial. However, we should have trivial Jones polynomial for $K$ since $<T>=(-a^{-3})^k(0,1)$.
This contradicts the assumption.\\

If $w=\sigma_2^{-\epsilon_0}\sigma_1$ for non-negative integer $\epsilon_0$, then $A=A_2^{-\epsilon_0}A_1$.
By using Lemma 3.2, we have\\

 $(f(a),g(a))=\left(a^{-\epsilon_0-1},\displaystyle{{a^{-\epsilon_0+1}+(-1)^{-\epsilon_0}a^{3\epsilon_0+5}\over 1+a^4}}\right)\neq (-a^{-3})^k(0,1)$ for any $k$. 
\\

Similarly, if $w=\sigma_2^{-\epsilon_0}\sigma_1^{-1}$ for non-positive integer $\epsilon_0$ then we have\\

$(f(a),g(a))=\left(a^{-\epsilon_0+1},\displaystyle{{a^{-\epsilon_0+3}+(-1)^{-\epsilon_0}a^{3\epsilon_0-1}\over 1+a^4}}\right)\neq (-a^{-3})^k(0,1)$ for any $k$. \\

If $w=\sigma_2^{-\epsilon_0}\sigma_1^{\epsilon_1}$ for $\epsilon_1\geq 2$, then 
we have $p_4(w)$ which represents a non-trivial alternating link. So, this violates the assumption too.\\

If $w=\sigma_2^{-\epsilon_0}\sigma_1^{\epsilon_1}\sigma_2^{-\epsilon_2}\sigma_1^{\epsilon_3}\cdot\cdot\cdot\sigma_1^{\epsilon_{2n-1}}$ for some positive (negative) integers $\epsilon_i$ ($1\leq i\leq 2n-1,~n\geq 2$) and non-negative (non-positive) integer $\epsilon_0$, then we have the reduced alternating presentation $v=\sigma_1^{\epsilon_1}\sigma_2^{-\epsilon_2}\cdot\cdot\cdot\sigma_1^{\epsilon_{2n-1}}$ for $p_4(v)$ for some positive (negative) integers $\epsilon_i$ ($1\leq i\leq 2n-1,~n\geq 2$).\\

This implies that $M(X_K)-m(X_K)\geq 4$ by Theorem~\ref{T4}.\\

However, if $(f(a),g(a))=(0,(-a^{-3})^k)$ then $M(X_K)-m(X_K)=0$ by Theorem~\ref{T4}.\\

This also contradicts the assumption.\\

Therefore, if  $<T>=(-a^{-3})^k<T_\infty>$ then $\mathbb{T}\approx \mathbb{T}_\infty$.\\

This completes the proof.

 \end{proof}
 
\begin{Cor}\label{T6}
Suppose that $T$ and $T'$ are the projections onto the $xy$-plane of two rational 2-tangles $\mathbb{T}$ and $\mathbb{T}'$ in $B^3$ so that $<T>=f(a)<T_0>+g(a)<T_{\infty}>$ and $<T'>=f'(a)<T_0>+g'(a)<T_{\infty}>$ respectively.\\
 
 Then, $\mathbb{T}\approx \mathbb{T}'$ if and only if $(f'(a),g'(a))=(-a^{-3})^k(f(a),g(a))$ for some integer $k$.
\end{Cor}

\begin{proof}

Let $q_4(w)$ be the plat presentation for $T$ and $q_4(v)$ be the  plat presentation for $T'$.

\begin{figure}[htb]
\includegraphics[scale=.9]{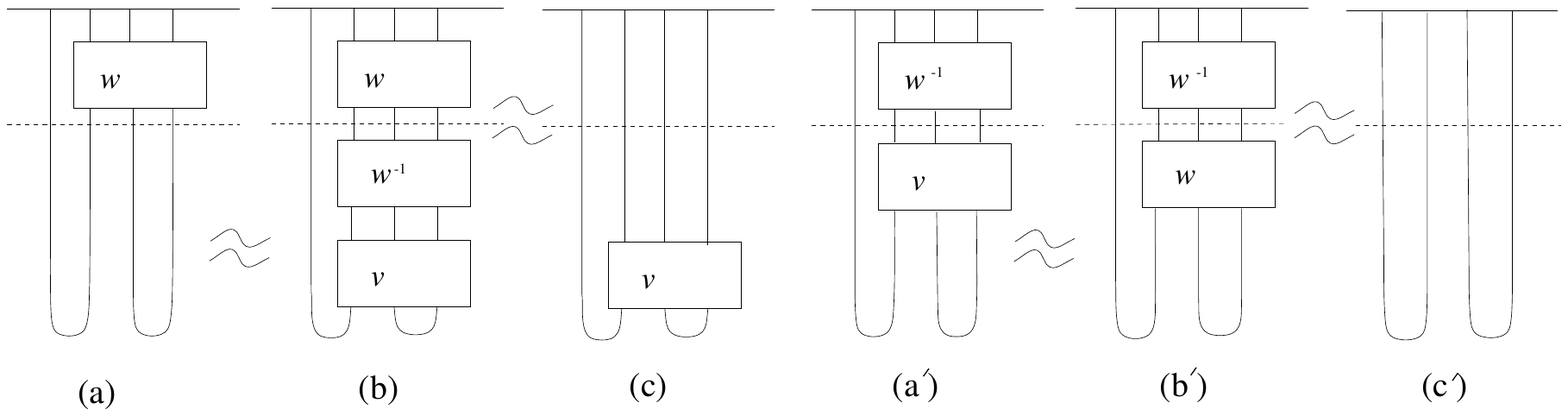}
\vskip -630pt
\caption{
\label{p5}}

\end{figure}
Now, consider $e=w^{-1}w$ and $w^{-1}v$.\\

Let $\mathbb{T}_1$ and $\mathbb{T}_1'$ be the rational 2-tangles of the  plat presentations $q_4(e)$ and $q_4(w^{-1}v)$ respectively. We note that $\mathbb{T}_1\approx\mathbb{T}_{\infty}$.\\

Then by Theorem~\ref{T5}, $ \mathbb{T}_1'\approx \mathbb{T}_1\approx\mathbb{T}_{\infty}$ if and only if $<T_1'>=(-a^{-3})^k<T_\infty>$ for some $k$.\\

Now, we claim that  $<T_1'>=(-a^{-3})^k<T_\infty>$ for some $k$ if and only if $<T'>=(-a^{-3})^{k'}<T>$ for some $k'$. To prove this, we repeatedly use the three axioms to remove the crossings below the dotted line of the diagram $(a)$ and the diagram $(b)$ respectively to have  $<T_1'>=(-a^{-3})^k<T_\infty>$ for some $k$. Then we repeatedly use the three axioms to remove the crossings above the dotted line of the diagram $(a)$ and the diagram $(b)$ respectively. (Refer to the diagram $(a)$ and $(b)$ of Figure~\ref{p5}.) Since $v=ww^{-1}v$, we note that if $<T_1'>=(-a^{-3})^k<T_\infty>$ for some $k$ then $<T'>=(-a^{-3})^{k'}<T>$ for some $k'$.
Similarly, we note that if $<T'>=(-a^{-3})^k<T>$ for some $k$ then $<T_1'>=(-a^{-3})^{k'}<T_\infty>$ for some $k'$ by using the diagrams $(a'),(b')$ and $(c')$ since $e=w^{-1}w=w^{-1}v$.\\

Now, it is enough to show that  $\mathbb{T}_1\approx \mathbb{T}_1'$ if and only if $\mathbb{T}\approx \mathbb{T}'$. This is also proved by the diagrams of Figure~\ref{p5}.\\

Therefore, $\mathbb{T}\approx \mathbb{T}'$ if and only if $<T'>=(-a^{-3})^{k'}<T>$ for some $k'$.\\

This completes the proof.

\end{proof}

\section{A new invariant of rational 3-tangles}\label{s4}

\begin{figure}[htb]
\includegraphics[scale=.45]{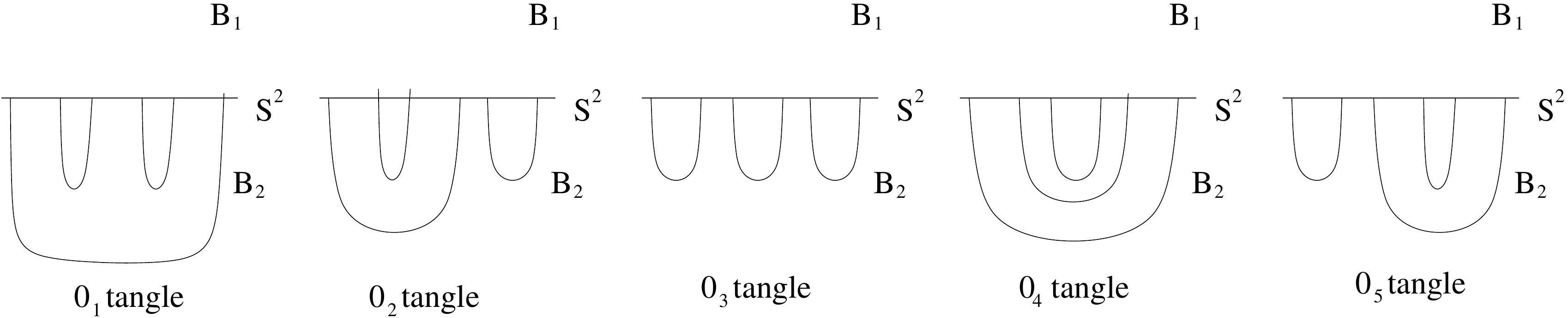}
\caption{}
\label{p6}
\end{figure}

Now, consider a rational $3$-tangle $\mathbb{T}$. Let $T$ be a rational $3$-tangle diagram.\\

H. Cabrera-Ibarra [3] defined the bracket polynomial of the rational 3-tangle diagram $T$ of $\mathbb{T}$ as $<T>=f^T_1(a)<{0_1}>+f^T_2(a)<{0_2}>+f^T_3(a)<{0_3}>+f^T_4(a)<{0_4}>+f^T_5<{0_5}>$, where 
$f^T_i(a)$ are Laurent polynomials in $a$ and $a^{-1}$ that are obtained by starting with $T$ and using the three axioms repeatedly until only the five trivial tangle diagrams $<{0_j}>$ in the expression given for $T$ are left. (See Figure~\ref{p6}.)   \\

Then, we define the vector $v_T=(f^T_1(a),f^T_2(a),...,f^T_5(a))$ for a rational 3-tangle diagram $T$. Then we note that the vector $v_T$ is an invariant under regular isotopy of $T$. Especially we have the following theorem.

\begin{Thm}\label{T7} If $\mathbb{T}\approx \mathbb{T}'$  then $v_T=(-a^{-3})^kv_{T'}$ for some $k$, where $T$ is a tangle diagram of  $\mathbb{T}$ and $T'$ is a tangle diagram of $\mathbb{T}'$. 
\end{Thm}

\begin{proof}
It is the generalization of the proof for the one direction of Theorem~\ref{T5}.
\end{proof}

\begin{figure}[htb]
\includegraphics[scale=.7]{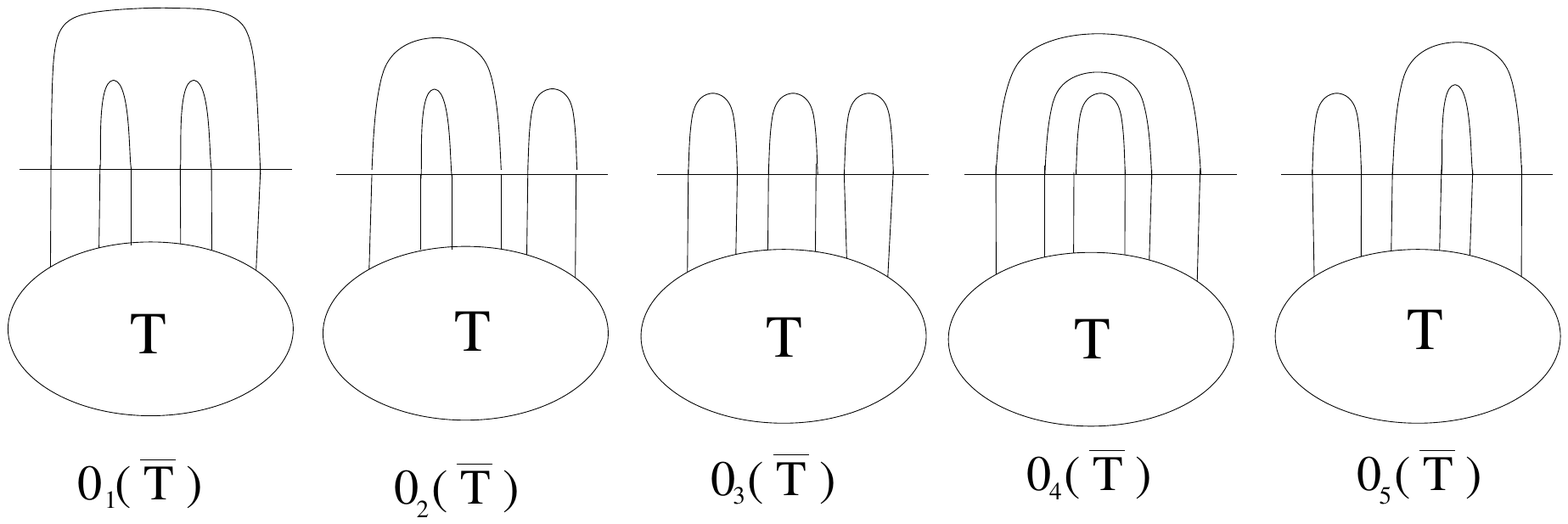}
\vskip -450pt
\caption{}
\label{p7}
\end{figure}

A link diagram $L$ is the $0_i$-$\emph{closure}$ of a rational 3-tangle diagram $T$, denoted by $0_i(\overline{T})$, if $L$ is obtained from $T$ by connecting the six endpoints of $T$ as the pattern shown above. (See Figure~\ref{p7}.)\\

\begin{figure}[htb]
\includegraphics[scale=.8]{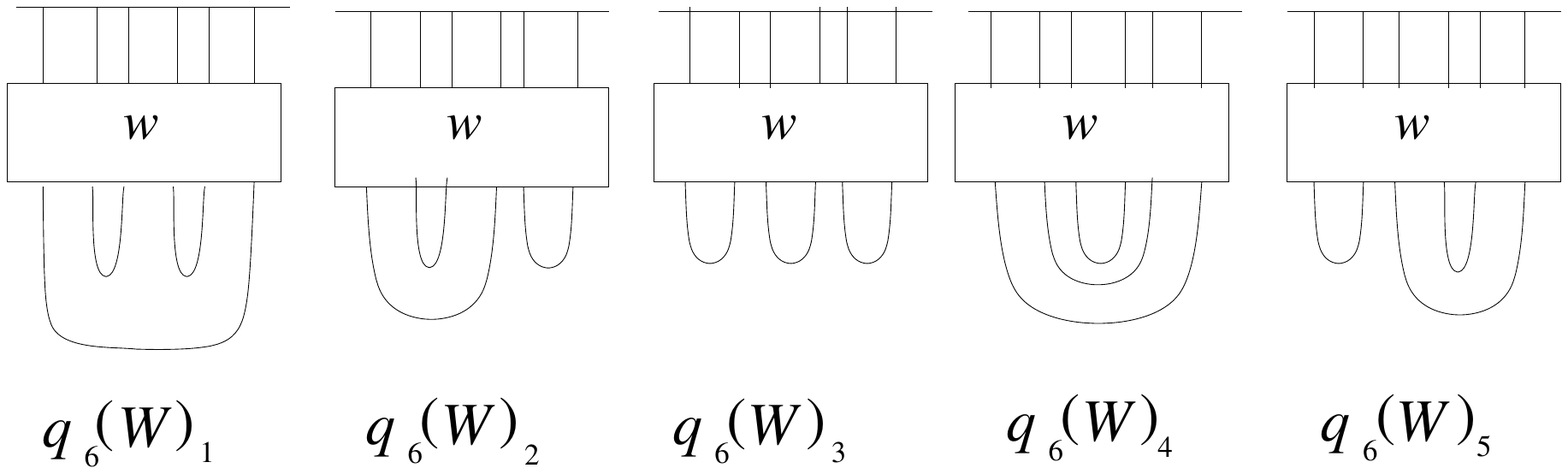}
\vskip -550pt
\caption{}
\label{p8}

\end{figure}

Suppose that $T$ has a $6$-plat presentation $q_6(w)$. 
Then, let $w=\sigma_{k_1}^{\epsilon_1}\sigma_{k_2}^{\epsilon_2}\cdot\cdot\cdot
\sigma_{k_{n-1}}^{\epsilon_{n-1}}\sigma_{k_{n}}^{\epsilon_{n}}$ for some non-zero integers $\epsilon_i$ ($1\leq i\leq n$), where $k_i\in\{1,2,3,4,5\}$.\\

Now, we define $q_6(w)_i$ which is a $6$-plat presentation of a rational $3$-tangle $T$ under the connectivity pattern of bottom endpoints induced by $0_i$ as in Figure 8. For example, $q_6(w)_3=q_6(w)$.
Then we say that a rational 3-tangle $\mathbb{T}$ is a $\emph{6-plat  tangle}$  if $\mathbb{T}$ is isotopic to a tangle $\mathbb{T}'$ and the projection of $\mathbb{T}'$ onto the $xy$-plane is a  $6$-plat presentation $q_6(w)_i$ for some $i$. Then, we say that a 6-plat tangle diagram is reduced alternating if the $6$-plat presentation is reduced alternating. \\

We note that each rational 3-tangle is a $6$-plat tangle. (Refer to~\cite{12} and~\cite{9}.) However, the set of all reduced alternating $6$-plat tangle diagrams is a proper subset of the collection of reduced alternating rational $3$-tangle diagrams. \\

In order to cover all reduced alternating rational $3$-tangle diagrams, we  define $\sigma_6$ as in Figure~\ref{p20}. Generally, the defined generators $\sigma_i$ ($1\leq i\leq 5)$ also can be defined from the twisting obtained by the extension to $B^3$ of the half Dehn twists between adjacent endpoints on $\partial B^3$. (Refer to~\cite{12}.)\\

\begin{figure}[htb]
\includegraphics[scale=.6]{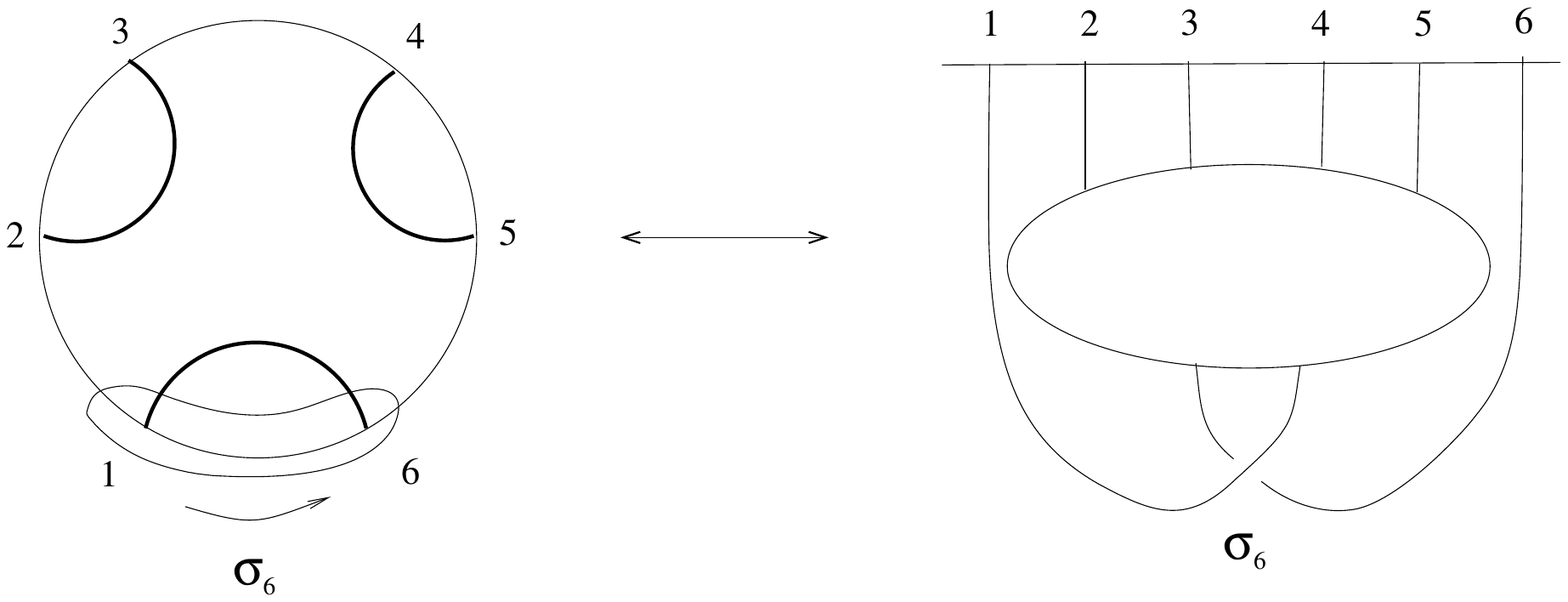}
\vskip -380pt
\caption{}
\label{p20}

\end{figure}

Let $|T|$ be the minimal crossing number of the diagram $T$. 
We say that a crossing \emph{is closest to} $S^2$ if at least two arcs of the four end arcs from the crossing directly meet the horizontal line which stands for $S^2$ without any crossing. Then we have the following lemma.
\begin{Lem}\label{T8}
Suppose that $T$ is a reduced alternating rational $3$-tangle diagram. If $|T|\geq 2$ then $0_i(\overline{T})$ is a reduced alternating link diagram for some $i$.
\end{Lem}
\begin{proof}

\begin{figure}[htb]
\includegraphics[scale=.8]{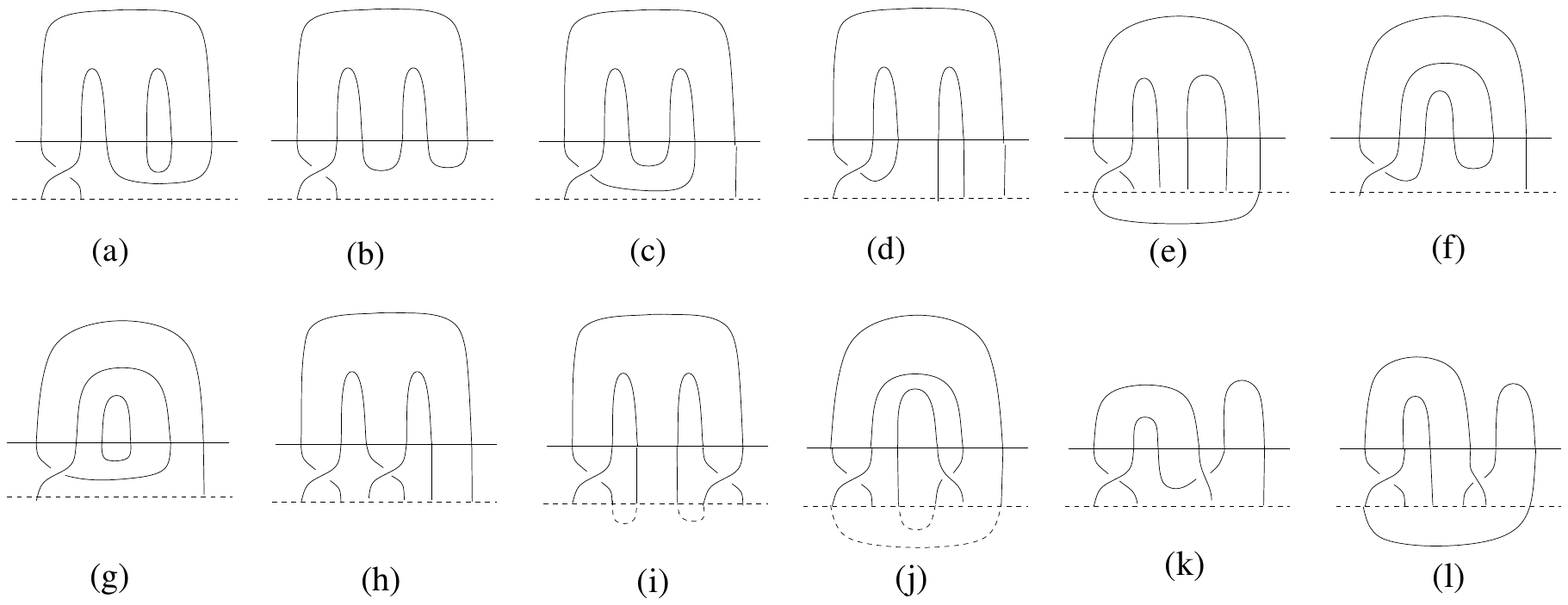}
\vskip -500pt
\caption{}
\label{p9}
\end{figure}

First of all, we claim that if a crossing of $0_i(\overline{T})$ is not a closest crossing then the crossing cannot be removed by the first Reidemeister move.
If a crossing is not closest one and the crossing is removed by Type I Reidemeister move, then the self crossing should be removed below $S^2$. However, $T$ is reduced alternating. This shows this claim.\\

 Suppose that $T$ has a closest crossing to $S^2$ which is obtained by $\sigma_1^{\pm 1}$ as in Figure~\ref{p9}. \\

First assume that $T$ does not have any other closest crossing except the crossing. \\

Consider $0_1$-closure of $T$. Then there are five cases $(a)-(e)$ which are not reduced alternating.\\

We note that the cases $(a),(b)$ and $(c)$ are impossible since $|T|\geq 2$ and $T$ is a reduced alternating rational $3$-tangle diagram. In the case $(d)$ of Figure~\ref{p9}, take $0_3$-closure of $T$ instead of $0_1$-closure.
Then, we have the diagrams $(f)$ and $(g)$ to have the crossing removed by the first Reidemeister move. However, they are also impossible since $|T|\geq 2$  and $T$ is a reduced alternating rational $3$-tangle diagram. Therefore, $0_3(\overline{T})$ should be reduced alternating. Similarly, in the case $(e)$ of Figure~\ref{p9}, we can check that $0_2(\overline{T})$ should be reduced alternating.\\

Therefore, $0_i(\overline{T})$ is reduced alternating for some $i$.\\

Now, assume that $T$ does have another closest crossing to $S^2$ possibly obtained by $\sigma_3^{\pm 1}$ or $\sigma_5^{\pm 1}$.  \\

Then consider $0_1(\overline{T})$ as in the diagram $(h)$ and $(i)$ of Figure~\ref{p9}.\\

In the diagram $(h)$, take $0_2$-closure to $T$. Then $0_2(\overline{T})$ should be  reduced alternating since  $T$ is a reduced alternating rational $3$-tangle diagram. If we have a case $(i)$, then
take $0_4$-closure of $T$. 
Then, we also note that $0_4(\overline{T})$ is a reduced alternating link diagram.\\

It is clear that $0_2(\overline{T})$ is reduced alternating if the three crossings by $\sigma_1^{\pm 1}, \sigma_3^{\pm1}$ and $\sigma_5^{\pm 1}$ are closest crossings to $S^2$.\\

At last, we assume that $T$ does have another closest crossing to $S^2$ obtained by $ \sigma_4^{\pm 1}$ as the diagram $(j)$. In order to have a crossing removed by the first Reidemeister move, at least one of the dotted arcs is a real arc. However, by taking $0_2$-closure of $T$ as in the diagrams $(k)$ and $(l)$, we note that the diagrams are impossible since $T$ is reduced alternating. So, we can eliminate this case as well.\\

Therefore, we just show that if $T$ has a closest crossing to $S^2$ which is obtained by $\sigma_1^{\pm 1}$ then $0_i(\overline{T})$ is reduced alternating for some $i$.\\

\begin{figure}[htb]
\includegraphics[scale=.9]{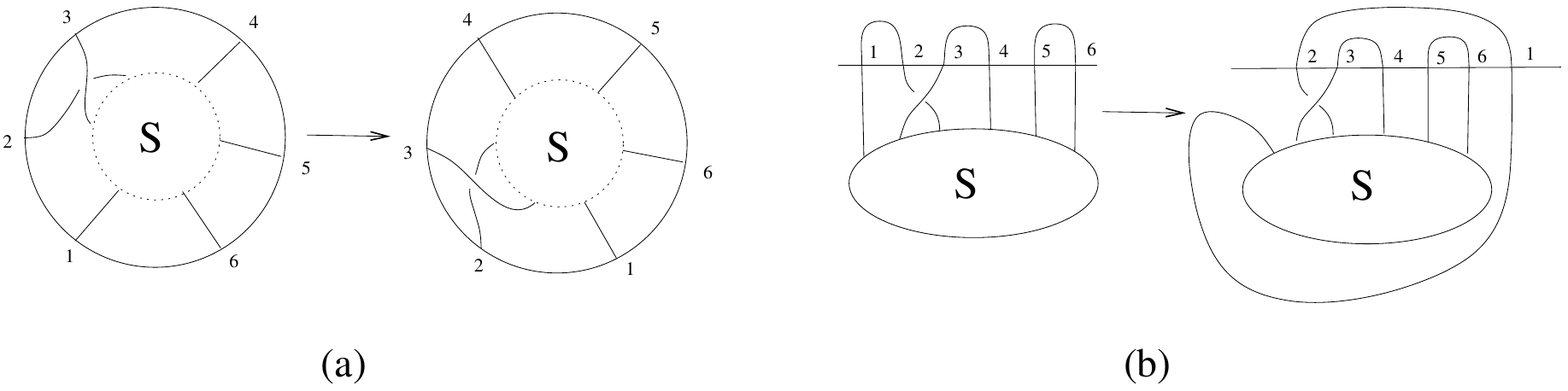}
\vskip -630pt
\caption{}
\label{p10}
\end{figure}

If $T$ has a closest crossing to $S^2$ which is obtained by $\sigma_2^{\pm 1}$, we modify the previous diagrams by rotating $60^\circ$ counterclockwise in the disk model  as the diagram $(a)$ of Figure~\ref{p10}. Then this case also can be proved by the previous arguments. (Refer to the diagram $(b)$ of Figure~\ref{p10}.)\\

Similarly, for the rest of  cases to prove this lemma, we modify the previous cases by rotating a multiple of $60^\circ$ in the disk model of tangle diagrams. \\

This completes the proof.\\

\end{proof}

\begin{Lem}\label{T9}
If $|T|=1$ Then $X_{0_i(\overline{T})}=(-a^2-a^{-2})^{(t-1)}$, where $i\in\{1,2,3,4,5\}$ and $t$ is the number of components of $0_i(\overline{T})$. 
\end{Lem}
\begin{proof}
Consider the diagram in Figure~\ref{p11}.

\begin{figure}[htb]
\includegraphics[scale=.5]{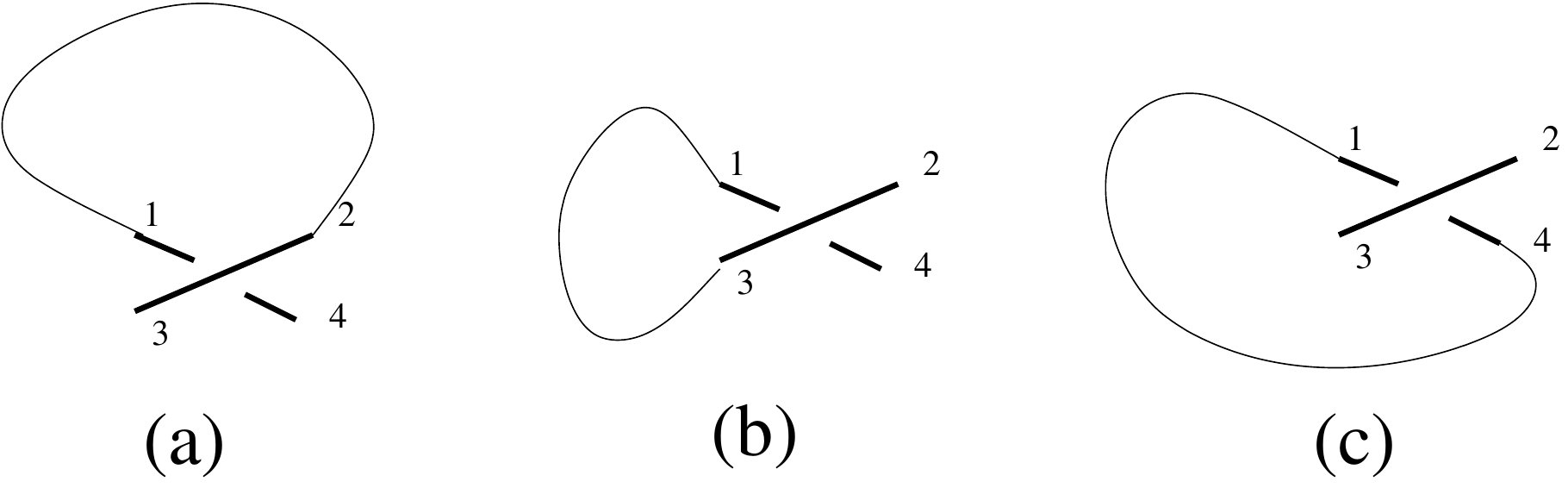}
\caption{}
\label{p11}
\end{figure}

Since $T$ has only one crossing, the cases $(a)$ and $(b)$ in Figure~\ref{p11} are possible. If we have a diagram $(c)$ of Figure~\ref{p11}, it is impossible to connect $2$ and $3$.
Also, the only crossing of $T$ should be removed by the first Reidemeister move. This completes the proof of Lemma 4.3.

\end{proof}

\begin{Lem}\label{T10}
Suppose that $T$ and $T'$ are reduced alternating rational $3$-tangle diagrams.\\
If $|T|=1$ and $|T'|\geq 2$, then $v_T\neq v_{T'}$.
\end{Lem}

\begin{proof}
First, we note that $0_i(\overline{T})$ has at most two trivial links (two split components) since $|T|=1$ for all $i$. (Refer Figure~\ref{p11}.) By Theorem~\ref{T4}  and the third axiom for the Kauffman bracket, we note that $M(X_{0_i(\overline{T})})-m(X_{0_i(\overline{T})})\leq 4$.\\

We also note that, by Lemma~\ref{T8}, there exists $0_j$-closure of $T'$ so that $0_j(\overline{T'})$ is a reduced alternating link diagram. Therefore, the minimal crossing number of $0_j(\overline{T'})$ for some $j$ is greater than or equal to $2$. This implies that $M(X_{0_j(\overline{T'})})-m(X_{0_j(\overline{T'})})\geq 8$.  However, if  $v_T= v_{T'}$, then $4\geq M(X_{0_j(\overline{T})})-m(X_{0_j(\overline{T})})=  M(X_{0_j(\overline{T'})})-m(X_{0_j(\overline{T'})})\geq 8$ since the writhes of $\overline{T}$ and $\overline{T'}$ cannot change $M(X_{0_j(\overline{T})})-m(X_{0_j(\overline{T})})$ and $M(X_{0_j(\overline{T'})})-m(X_{0_j(\overline{T'})})$. It makes a contradiction.\\

Therefore, $v_T\neq v_{T'}$.

\end{proof}

We remark that Eliahou and Kauffman~\cite{4} found an infinite number of $2$-component links with Kauffman polynomial equal to $-a^2-a^{-2}$. However, every element of the families is not (reduced) alternating by Theorem~\ref{T4}.

\begin{Lem}\label{T11}
Suppose that $\mathbb{T}$ and $\mathbb{T'}$ are rational $3$-tangles and
 $T$ and $T'$ are rational $3$-tangle diagrams of $\mathbb{T}$ and $\mathbb{T'}$ respectively. If $T$ and $T'$ are reduced alternating rational $3$-tangle diagrams  and $|T|\leq |T'|\leq 1$, then $\mathbb{T}\approx \mathbb{T}'$ if and only if $v_T=v_{T'}$.
\end{Lem}

\begin{proof}
We note that  $v_{T}, v_{T'}\in\{(1,0,0,0,0),(0,1,0,0,0,),...,(0,0,0,0,1)\}$ if $|T|= |T'|=0$. They are distinguished by the trivial rational $3$-tangle type.\\

We also note that  $v_{T}, v_{T'}\in\{(a^{\pm 1},a^{\mp 1},0,0,0),(a^{\pm 1},0,0,a^{\mp 1},0),(a^{\pm 1},0,0,0,a^{\mp 1}),(0,a^{\pm 1},a^{\mp 1},$ $0,0),(0,0,a^{\pm 1},a^{\mp 1},0),(0,0,a^{\pm 1},0,a^{\mp 1})\}$ if $|T|= |T'|= 1$. They  are also distinguished by the rational $3$-tangle type with $|T|=|T'|=1.$\\

This completes the proof.

\end{proof}

\begin{Lem}\label{T12}
Suppose that $T$ and $T'$ are reduced alternating rational $3$-tangle diagrams with  $v_T=(-a^{-3})^kv_{T'}$ for some $k$ and $|T|\geq|T'|\geq 2$.
Then both $T$ and $T'$ are positive (negative) reduced alternating, where positive (negative) alternating diagram means that the crossings of the diagram are obtained only by $\sigma_1,\sigma_3,\sigma_5$, $\sigma_2^{-1},\sigma_4^{-1}$ or $\sigma_6^{-1}$ ($\sigma_1^{-1},\sigma_3^{-1},\sigma_5^{-1},\sigma_2,\sigma_4$ or $\sigma_6$).
\end{Lem}

\begin{proof}

For a contradiction, assume that $T$ is positive  reduced alternating and $T'$ is negative  reduced alternating.
Then, consider the diagrams of Figure~\ref{p21}.\\

We have $T_1'$ by having crossings in the cylinder $S$ so that  $T_1'$ is isotopic to a trivial tangle diagram as in the second diagram of Figure~\ref{p21}. We note that the crossings in $S$ are obtained only by $\sigma_1,\sigma_3,\sigma_5$, $\sigma_2^{-1},\sigma_4^{-1}$ or $\sigma_6^{-1}$. \\

Then attach the same cylinder $S$ to $T$ to have $T_1$ as in Figure~\ref{p21}. We note that  $v_{T_1}=(-a^{-3})^kv_{T_1'}$ for some $k$, where $v_{T_1'}\in\{(1,0,0,0,0),(0,1,0,0,0,),...,(0,0,0,0,1)\}$.\\

\begin{figure}[htb]
\includegraphics[scale=.7]{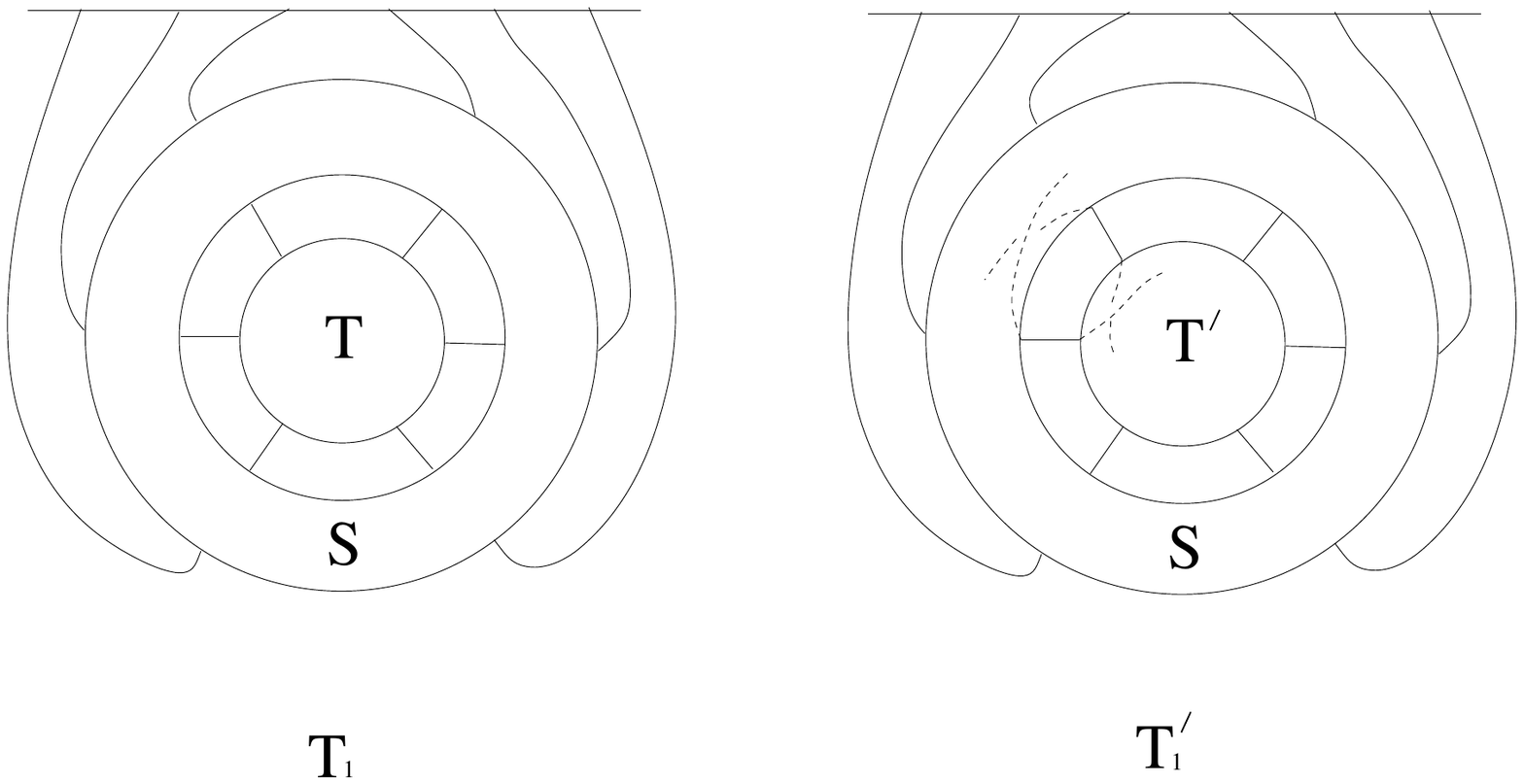}
\vskip -400pt
\caption{}
\label{p21}
\end{figure}

Clearly, the minimal crossing number of $T_1'$ is zero. If $T$ is a connected tangle diagram, then $|T_1|=|T|+|T'|\geq 4$ since  $T_1$ is a reduced alternating rational $3$-tangle diagram. So, $|T_1|\geq 4$.\\

We note that $0_i(\overline{T_1'})$ has at most three trivial components. So, $M(X_{0_i(\overline{T_1'})})-m(X_{0_i(\overline{T_1'})})\leq 8$ for all $i$.
By Lemma~\ref{T8}, there exists a $0_k$-closure of $T_1$ so that $0_k(\overline{T_1})$ is a reduced alternating link diagram with $|0_k(\overline{T_1})|\geq 4$. So, $M(X_{0_k(\overline{T_1})})-m(X_{0_k(\overline{T_1})})\geq 16$ for some $k$.\\

If $v_{T_1}= (-a^{-3})^lv_{T_1'}$ for some integer $l$, then we note that  $X_{0_k(\overline{T_1})}=(-a^{-3})^s X_{0_k(\overline{T_1'})}$ for some $s$ by considering the writhes of $T_1$ and $T_1'$. This implies that  $16\leq M(X_{0_k(\overline{T_1})})-m(X_{0_k(\overline{T_1})})=M(X_{0_k(\overline{T_1'})})-m(X_{0_k(\overline{T_1'})})\leq 8$. This makes a contradiction. Therefore, $v_{T_1}\neq (-a^{-3})^lv_{T_1'}$ for any integer $l$.\\

If $T$ is not a connected tangle diagram, then the only one possible case to have the condition $v_{T_1}= (-a^{-3})^lv_{T_1'}$ for some integer $l$ is when $2\leq |T'|\leq |T|\leq 3$ since if $|T|>3$ then $|T_1|\geq 4$. So, we note that $2\leq |T_1|\leq 3$. Then there exist $i$ so that $0_i(\overline{T_1})$ is a reduced alternating by Lemma~\ref{T8}. However, there is no such link $L$ so that $|L|=2$ or $3$ and $X_L=(-a^2-a^{-2})^{m}$ for some integer $m$. (Refer to  KnotInfo by Livingston/Cha.)  Therefore, $v_{T_1}\neq (-a^{-3})^lv_{T_1'}$ for any integer $l$.\\

Both cases contradict the assumption that $v_T=(-a^{-3})^lv_{T'}$ for some $l$ and this completes the proof.
\end{proof}

\begin{Lem}\label{T20}
Suppose that $T$ and $T'$ are reduced alternating rational $3$-tangle diagrams.\\
  If $v_T=(-a^{-3})^kv_{T'}$ for some $k$, then $|T|=|T'|$.
\end{Lem}
\begin{proof}
Suppose that $|T|\geq|T'|\geq 2$. The other cases are clear by Lemma~\ref{T10} and Lemma~\ref{T11}. 
Assume that $|T|>|T'|$ for a contradiction. We note $X_{0_i(\overline{T})}= 
(-a^{-3})^kX_{0_i(\overline{T'})}$ for all $i$ and some $k$. Otherwise $v_T\neq(-a^{-3})^kv_{T'}$. We also know that there exists $j$ so that  $0_j(\overline{T})$ is a reduced alternating link diagram by Lemma~\ref{T8}.\\

 In order to have $X_{0_j(\overline{T})}= 
(-a^{-3})^kX_{0_j(\overline{T'})}$, we should have three conditions $(1)$ $|T|=|T'|+1$, $(2)$ $T$ is a connected tangle diagram and $(3)$ $T'$ has one separated arc. 
We note that $T'$ cannot have three separated trivial arcs since $|T'|\geq 2$. Also, $T'$ cannot be a connected tangle diagram. If $T'$ is a connected tangle diagram then 
we also have $M(X_{0_j(\overline{T})})-m(X_{0_j(\overline{T})})>M(X_{0_j(\overline{T'})})-m(X_{0_j(\overline{T'})})$ for some $j$ since $|T|>|T'|$. (Refer to Theorem $2$ of \cite{11}.)
This  implies that $X_{0_j(\overline{T})}\neq 
(-a^{-3})^kX_{0_j(\overline{T'})}$ for some $j$. So, $T'$ has one separated arc.
Then we note that $|T|=|T'|+1$. If $|T|>|T'|+1$ then we also have $M(X_{0_j(\overline{T})})-m(X_{0_j(\overline{T})})>M(X_{0_j(\overline{T'})})-m(X_{0_j(\overline{T'})})$ for some $j$ and it makes a contradiction.
At last, if $T$ is not a connected tangle diagram then there exist $j$ so that $0_j(\overline{T})$ is a reduced alternating link having two separated links. Then we conclude that $M(X_{0_j(\overline{T})})-m(X_{0_j(\overline{T})})>M(X_{0_j(\overline{T'})})-m(X_{0_j(\overline{T'})})$ for some $j$ since $|T|>|T'|$. Then we have $X_{0_j(\overline{T})}\neq
(-a^{-3})^kX_{0_j(\overline{T'})}$.

\begin{figure}[htb]
\includegraphics[scale=.4]{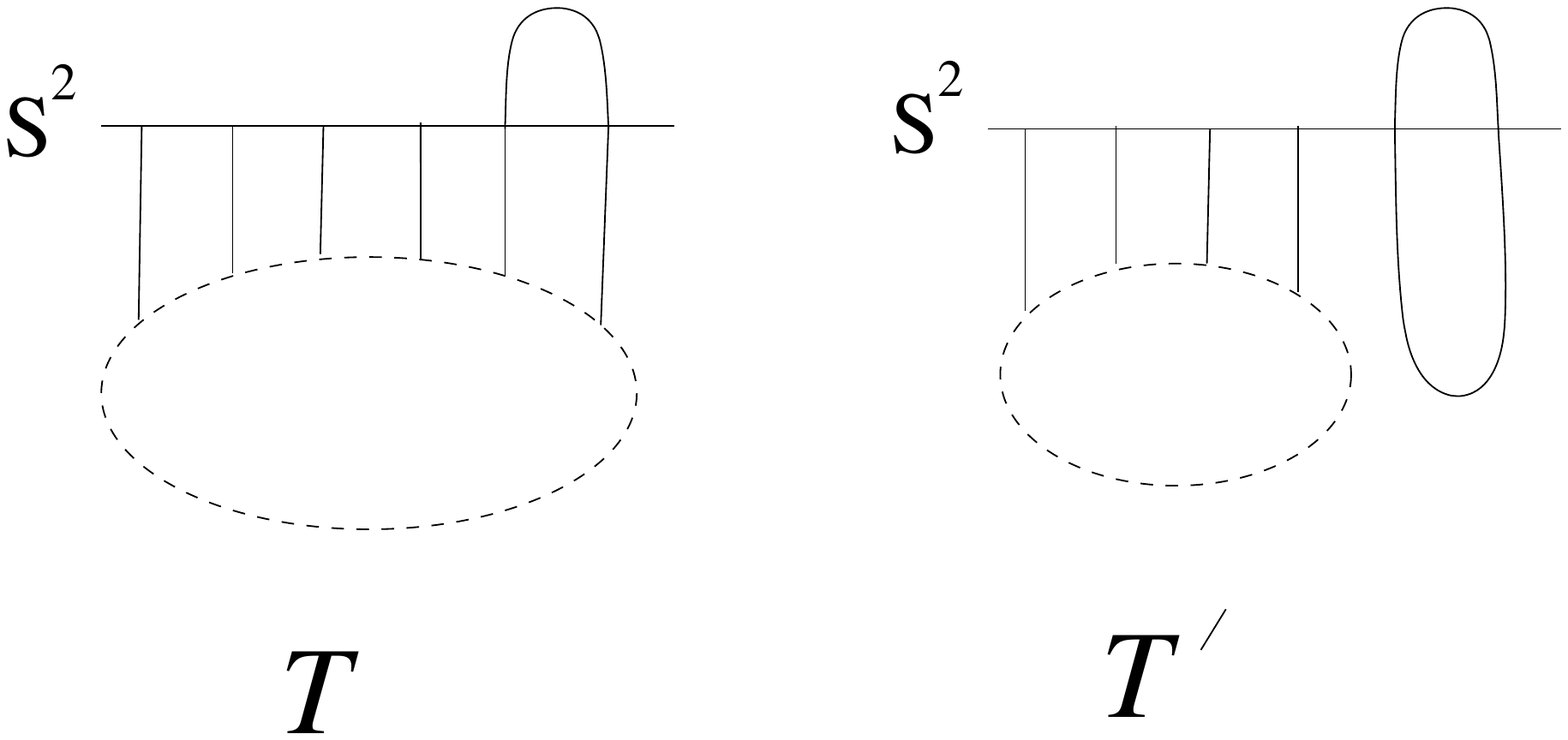}
\vskip -225pt
\caption{}
\label{p26}
\end{figure}

Now, we take $T$ and $T'$ so that $|T|+|T'|$ is minimal.
Now, consider the diagrams of Figure~\ref{p26}. The rest of cases are obtained from a multiple of $60^\circ$ rotation in the disk model which  changes the positions of the endpoints.

\begin{figure}[htb]
\includegraphics[scale=.7]{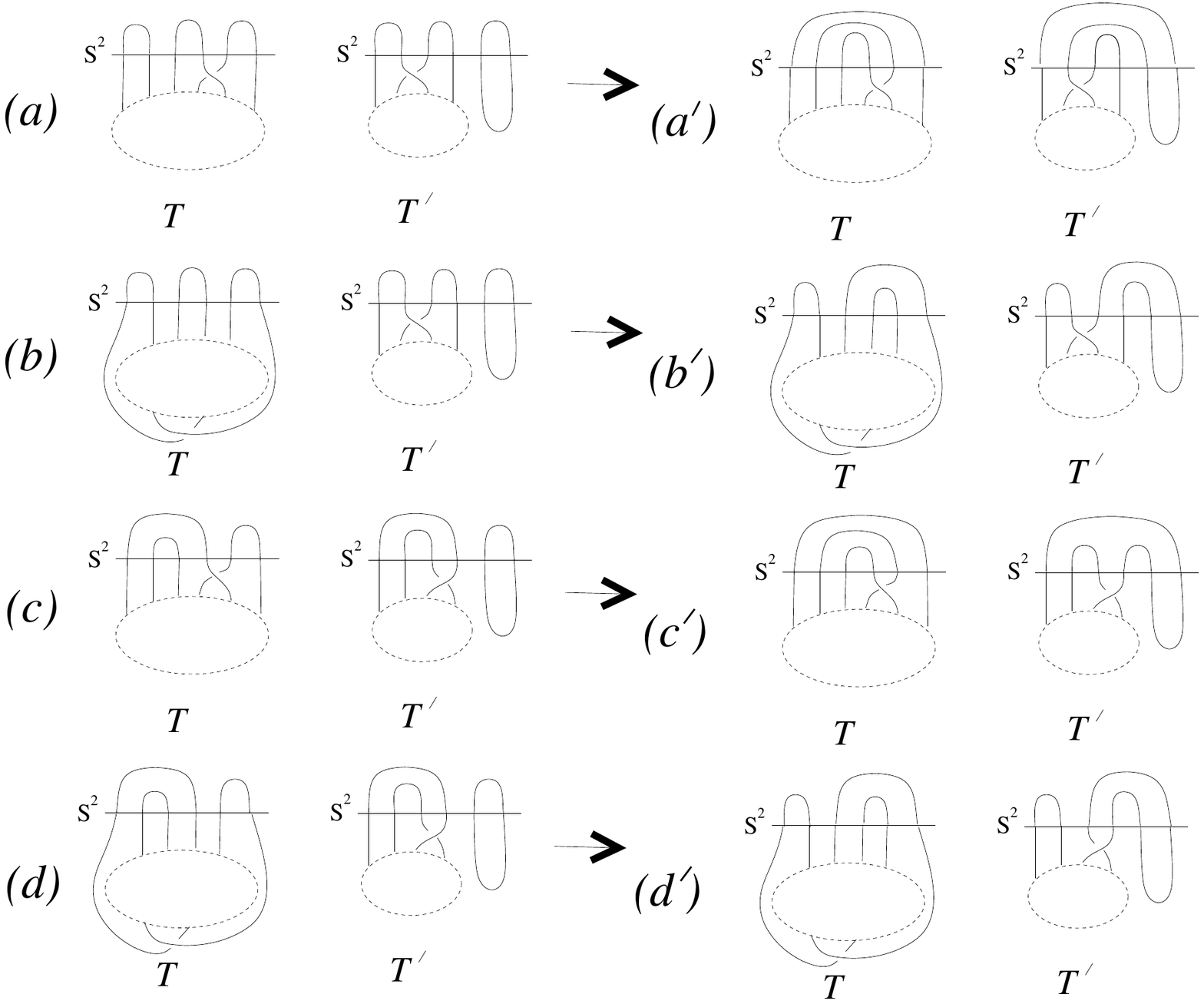}
\vskip -285pt
\caption{}
\label{p27}
\end{figure}

For the $0_j$-closure of $T$ to have a connected reduced alternating link, one of the connectivity is fixed as in the diagrams of Figure~\ref{p26}. Otherwise, $0_j(\overline{T'})$ is also  a connected link. This implies that $M(X_{0_j(\overline{T})})-m(X_{0_j(\overline{T})})>M(X_{0_j(\overline{T'})})-m(X_{0_j(\overline{T'})})$ for some $j$ since $|T|=|T'|+1$. So, it makes a contradiction.\\

Then we have only four possible cases $(a)-(d)$ for this as in Figure~\ref{p27}. Otherwise, either $0_j(\overline{T})$ is not reduced alternating or $|T|+|T'|$ is not minimal.  We also note that both $0_j(\overline{T})$ and $0_j(\overline{T})$ are reduced alternating.   The cases $(a)$ and $(c)$ assume that the given crossing in $T$ is a closest crossing to $S^2$. and the cases $(b)$ and $(d)$ assume that the given crossing in $T$ is the only one closest crossing to $S^2$.\\

With a similar argument in Lemma~\ref{T8}, we can show that the closure of $T$ in diagrams $(a')-(d')$ is also a reduced alternating link. Also, we see that the same closure of $T'$ in diagrams $(a')-(d')$ is now a connected link. By using the condition that $|T|=|T'|+1$, we have $M(X_{0_j(\overline{T})})-m(X_{0_j(\overline{T})})>M(X_{0_j(\overline{T'})})-m(X_{0_j(\overline{T'})})$ for some $j$. Therefore, 
$X_{0_j(\overline{T})}\neq
(-a^{-3})^kX_{0_j(\overline{T'})}$. This makes a contradiction.\\

So, it is impossible to have $T$ and $T'$ to satisfy the given three conditions.\\

 Therefore, we conclude that $|T|=|T'|$.
\end{proof}

\begin{figure}[htb]
\includegraphics[scale=.7]{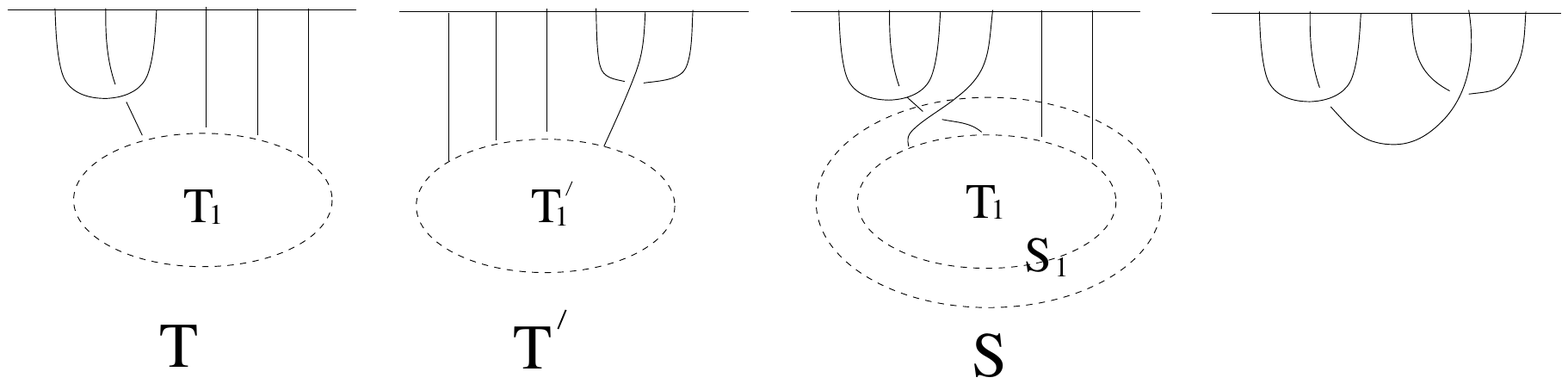}
\vskip -475pt
\caption{}
\label{p22}
\end{figure}

\begin{Lem}\label{T21}
Suppose that $\mathbb{T}$ and $\mathbb{T}'$ are rational $3$-tangles with reduced alternating rational $3$-tangle diagrams $T$ and $T'$  as in Figure~\ref{p22}, where $T_1$ and $T_1'$ are rational $2$-tangle diagrams.
If $v_T=(-a^{-3})^kv_{T'}$ for some $k$, then $\mathbb{T}\approx\mathbb{T}'$.
\end{Lem}

\begin{proof}
Let $<T_1>=f(a)<T_0>+g(a)<T_{\infty}>$ and $<T_1'>=f'(a)<T_0>+g'(a)<T_{\infty}>$.
Then for $A,B,C$ and $D$ of Figure~\ref{p23}, we have the relations $<T>=f(a)A+g(a)B$ and $<T'>=f'(a)C+g'(a)D$.\\

\begin{figure}[htb]
\includegraphics[scale=.5]{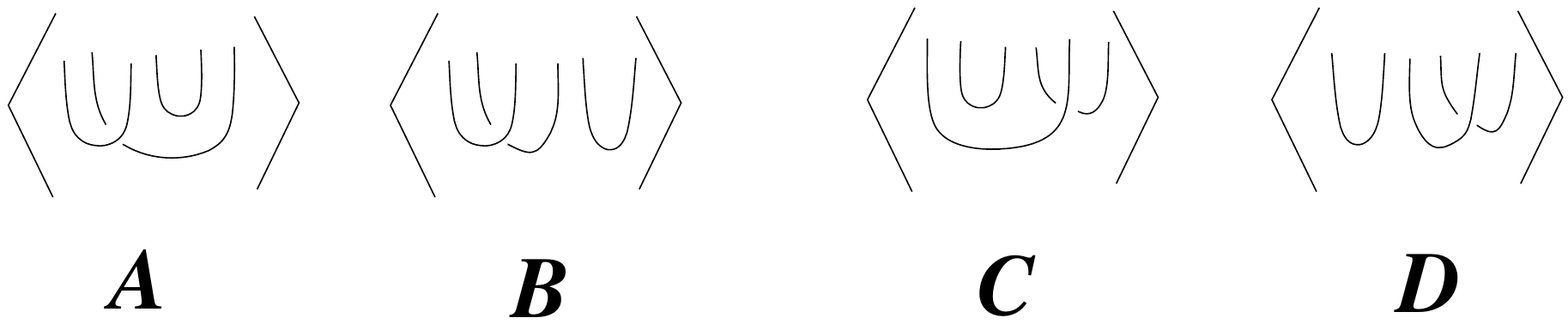}
\vskip -350pt
\caption{}
\label{p23}
\end{figure}

We also note that $A=a<0_5>+a^{-1}<0_1>, B=a<0_3>+a^{-1}<0_2>, C=a^{-1}<0_1>+a<0_2>$ and $ D=a^{-1}<0_5>+a<0_3>$.\\

Therefore, $<T>=f(a)A+g(a)B=f(a)(a<0_5>+a^{-1}<0_1>)+g(a)(a<0_3>+a^{-1}<0_2>)=a^{-1}f(a)<0_1>+a^{-1}g(a)<0_2>+ag(a)<0_3>+af(a)<0_5>$ and
$<T'>=f'(a)C+g'(a)D=f'(a)(a^{-1}<0_1>+a<0_2>)+g'(a)(a^{-1}<0_5>+a<0_3>)=a^{-1}f'(a)<0_1>+af'(a)<0_2>+ag'(a)<0_3>+a^{-1}g'(a)<0_5>$.\\

This implies that $g(a)=a^{2}f(a)$ and $(f(a),g(a))=(f'(a),g'(a))$. Actually, we will conclude that $T_1=T_1'$ which has only one crossing.\\

Now consider the rational $3$-tangle diagram $S$ which contains the rational $2$-tangle diagram $S_1$ as the third diagram of Figure~\ref{p22}, where the inner circle is for $T_1$ and the outer circle is for $S_1$.\\

We have $<S_1>=f(a)(a<T_{0}>+a^{-1}<T_{\infty}>+g(a)(-a^{-3})<T_{\infty}>=af(a)<T_0>+(a^{-1}f(a)-a^{-3}g(a))<T_{\infty}>=af(a)<T_{0}>$ since $f(a)=a^2g(a)$.\\

We note that the Kauffman bracket vector of $S_1$ is $(af(a),0)$.\\

In Theorem~\ref{T2}, if we consider the determinant of the matrix $A_{S_1}$ then we have $af(a)h(a)=(a^2)^l$ for some integer $l$ and a Laurent polynomial $h(a)$, where $A_{S_1}$ is the matrices product for calculating the bracket vector of $S_1$. This implies that $f(a)=a^n$ for some integer $n$. So, the Kauffman bracket vector of $S_1$ is $(a^{n+1},0)$. By a similar argument as in Theorem~\ref{T5}, we conclude that $|S_1|=0$.
So, both cases have the same diagram as the last diagram of Figure~\ref{p22}.\\

 This completes the proof.
\end{proof}

Now, we want to prove another direction of the main theorem.
\begin{Thm}\label{T13}
Suppose that $\mathbb{T}$ and $\mathbb{T}'$ are rational $3$-tangles with reduced alternating rational $3$-tangle diagrams $T$ and $T'$  respectively.  If $v_T=(-a^{-3})^kv_{T'}$ for some $k$, then $\mathbb{T}\approx \mathbb{T}'$.
\end{Thm}

\begin{proof}
Suppose that there exists  reduced alternating rational $3$-tangle diagrams $T$ and $T'$ so that $v_T=(-a^{-3})^kv_{T'}$ for some $k$ but $\mathbb{T}\not\approx \mathbb{T}'$.\\

 Then we choose a pair of $T$ and $T'$ satisfying the previous condition and  $|T|=|T'|$ is minimal.\\

 We note that both $T$ and $T'$ are positive (negative) reduced alternating by Lemma~\ref{T12}.\\

Also, we note that the closest crossings of $T$ and $T'$ to $S^2$ are obtained by different $\sigma_i$. Otherwise, by removing the common closest crossings of $T$ and $T'$, we can get $T_1$ and $T_1'$. Then we note that $v_{T_1}=(-a^{-3})^kv_{T_1'}$. So, it contradicts the assumption that $|T|=|T'|$ is minimal.\\

By Lemma~\ref{T10} and Lemma~\ref{T11}, if there exists a such example then  $|T|=|T'|\geq 2$.
\\

If any of $T$ and $T'$ is not a connected tangle, possibly $T$, then take $0_j$-closure of $T$ to have a reduced alternating link for some $j$ which has a separated trivial knot. Then we note that $0_j(\overline{T'})$ is also a reduced alternating link which has a separated trivial knot. Otherwise, we have $M(X_{0_j(\overline{T})})-m(X_{0_j(\overline{T})})>M(X_{0_j(\overline{T'})})-m(X_{0_j(\overline{T'})})$. It contradicts the condition that  $v_T=(-a^{-3})^kv_{T'}$ for some $k$
since $X_{0_j(\overline{T})}\neq
(-a^{-3})^kX_{0_j(\overline{T'})}$. So, it is enough to consider rational 2-tangles by ignoring the common separated trivial arc of $T$ and $T'$. Then by using Corollary~\ref{T6} we can show that if $v_T=(-a^{-3})^kv_{T'}$ for some $k$, then $\mathbb{T}\approx \mathbb{T}'$. \\

Now, we start to consider the possible cases of $T$.\\

Case 1: Assume that a crossing of $T$ by $\sigma_1^{\pm 1}$ is the only one closest crossing to $S^2$. \\
\begin{figure}[htb]
\includegraphics[scale=.85]{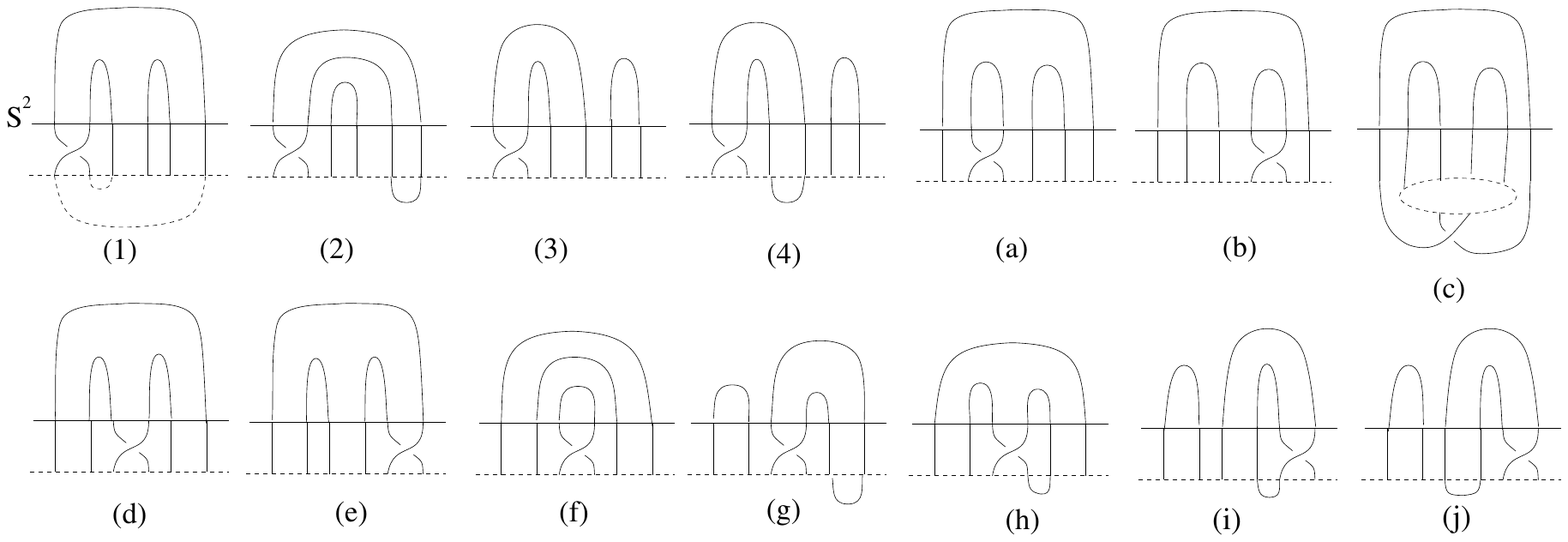}
\vskip -520pt
\caption{}
\label{p12}
\end{figure}

Then $0_1$-closure of $T$ makes a reduced alternating link diagram if the diagram $(1)$ of Figure~\ref{p12} is not the diagram of $T$ since $|T|\geq 2$ and $T$ is a reduced alternating tangle diagram.\\

First, assume that the diagram $(1)$ is not the diagram of $T$.\\

We note that  $0_1$-closure of $T'$ also needs to be reduced alternating link diagram if the diagram $(1)$ of Figure~\ref{p12} is not the diagram of $T$. Otherwise, $X_{0_i(\overline{T})}\neq (-a^{-3})^kX_{0_i(\overline{T'})}$ for any $k$ and it makes a contradiction.\\

We note that the diagrams $(a), (b)$ and $(c)$ are not reduced alternating. So, the closest crossing to $S^2$ of $T'$  is by either $\sigma_3^{\pm 1}$ or $\sigma_5^{\pm 1}$. (See $(d)$ and $(e)$.) \\

Suppose that the closest crossing to $S^2$ of $T'$ is by $\sigma_3^{\pm 1}$  as the diagram $(d)$ of Figure~\ref{p12}.\\

 Then we consider $0_4$-closure of $T$ as the diagram $(2)$ of Figure~\ref{p12}. If the diagram is a reduced alternating link diagram, then we have a contradiction because $0_4(\overline{T'})$ is not reduced alternating (See $(f)$). So, $T$ does not have any crossing with the rightmost string of $T$ as the diagram $(2)$.
Now, consider $0_5$-closure of $T$. Then it is clear that $0_5(\overline{T})$ is not reduced alternating link diagram. So, we note that $0_5(\overline{T'})$ is also not reduced alternating link diagram. Otherwise, $X_{0_i(\overline{T})}\neq (-a^{-3})^kX_{0_i(\overline{T'})}$ for any $k$ since $|T|=|T'|$ and it makes a contradiction. So, we have a diagram $(g)$ or $(h)$ for $T'$. Otherwise, $0_5(\overline{T'})$ is a reduced alternating link diagram, but $0_5(\overline{T})$ is not. We may consider the case that  $\sigma_5^{\pm 1}$ also can be a closest crossing to $S^2$ since we did not cover the case. However, in this case, we note that $0_5(\overline{T'})$ is reduced alternating. This makes a contradiction too.
If $(g)$ is the diagram of $T'$ then both $T$ and $T'$ have the rightmost string with no crossing. Also, the other two strings of each make a $4$-plat presentation since they make a rational $2$-tangle diagram. 
By Corollary~\ref{T6}, we note that $T=T'$ if $v_T=(-a^{-3})^kv_{T'}$ for some $k$.
If $(h)$ is the diagram of $T'$ then consider $0_1$-closure of $T'$. Then we see that $0_1(\overline{T'})$ is not reduced alternating link diagram, but $0_1(\overline{T'})$ is reduced alternating. This makes a contradiction.\\

Suppose that the closest crossing to $S^2$ of $T'$ is by $\sigma_5^{\pm 1}$ as the diagram $(e)$ of Figure~\ref{p12}.\\

 Then, we note that $0_2(\overline{T})$ is not reduced alternating since $0_2(\overline{T'})$ is not reduce alternating. So,  the second string of $T$ cannot have any crossing as the diagram $(4)$ of Figure~\ref{p12}.
 We note that  $0_5(\overline{T'})$ is not reduced alternating since $0_5(\overline{T})$ is not.  So, we have $0_5(\overline{T'})$ as the diagrams $(i)$ and $(j)$ of Figure~\ref{p12}.
 If we have the diagram $(i)$ then we note that $0_1(\overline{T'})$ is not reduced alternating. This makes a contradiction since $0_1(\overline{T})$ is reduced alternating. If we have the diagram $(j)$ then the second string of $T'$ also does not have any crossing. By Corollary~\ref{T6}, we know that  $T=T'$ if $v_T=(-a^{-3})^kv_{T'}$ for some $k$.
 \\
\begin{figure}[htb]
\includegraphics[scale=.6]{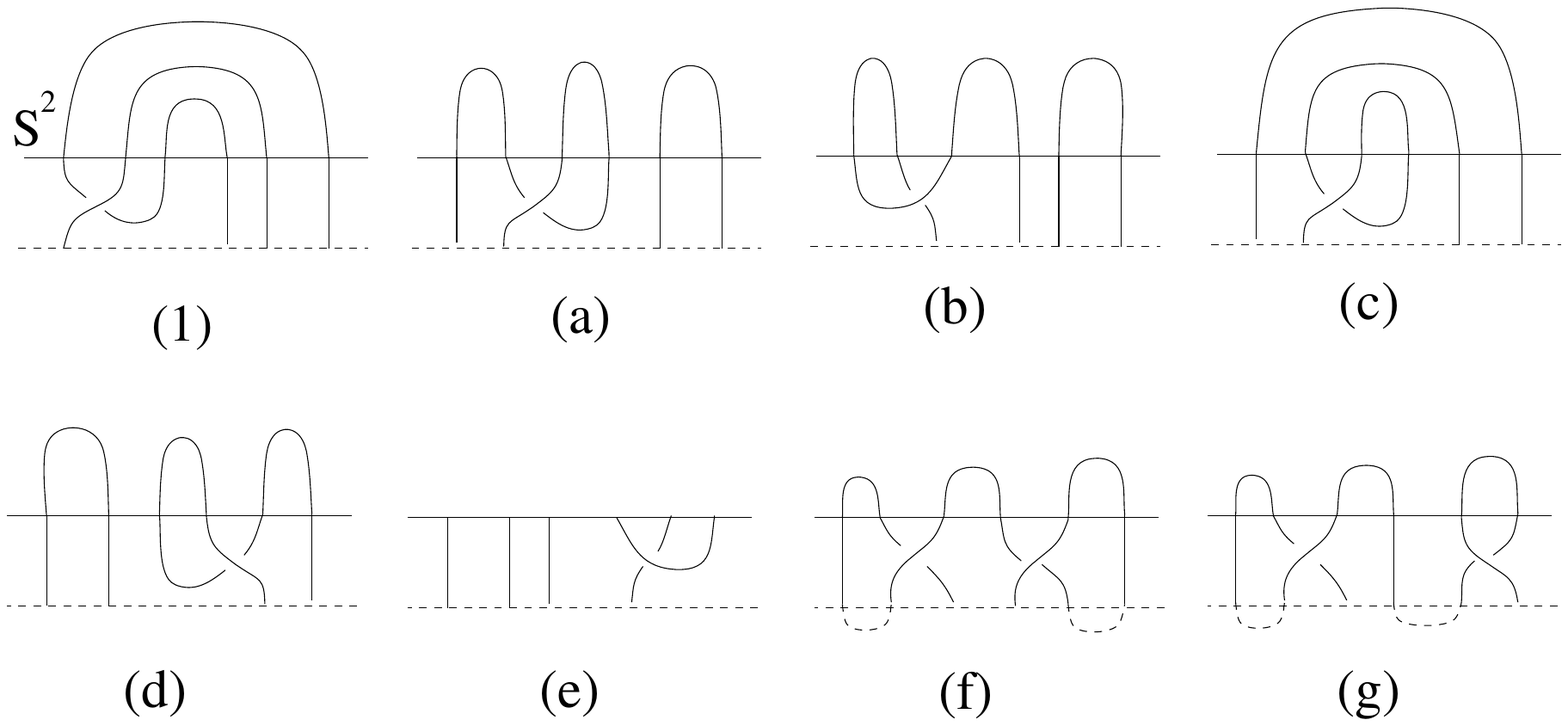}
\vskip -340pt
\caption{}
\label{p24}
\end{figure}

Now, assume that the diagram $(1)$ of Figure~\ref{p12} is the diagram of $T$. It is enough to consider the case that the smaller dotted arc is a real arc in the diagram $(1)$ since the other case is a symmetry case of this.\\

Then we note that $0_4$-closure of $T$ is now reduced alternating as the diagram $(1)$ in Figure~\ref{p24} since $|T|\geq 2$ and $T$ is reduced alternating. This implies that the closest crossing to $S^2$ of $T'$ is either by $\sigma_2^{\pm 1},\sigma_4^{\pm 1},\sigma_5^{\pm 1}$ or $\sigma_6^{\pm 1}$ since $0_4(\overline{T'})$ also should be reduced alternating. \\

First, assume that a crossing of $T'$ by $\sigma_2^{\pm 1}$ is the only one closest crossing to $S^2$.\\

Since $0_3(\overline{T})$ is not reduced alternating, we note that $0_3(\overline{T})$ is also not reduced alternating. So, the diagram of $T'$ is either $(a)$ of $(b)$ of Figure~\ref{p24}.
\\

If $(a)$ is the diagram of $T'$ then $0_4(\overline{T'})$ should be reduced alternating since $0_4(\overline{T})$ does. However, it is not reduced alternating.\\

$(b)$ is also not the diagram of $T'$ since the crossing by $\sigma_2^{\pm 1}$ is also a crossing by $\sigma_1^{\mp 1}$ in this case.\\

Now, assume that a crossing of $T'$ by $\sigma_4^{\pm 1}$ is the only one closest crossing to $S^2$.\\

Since $0_3(\overline{T})$ is not reduced alternating, $0_3(\overline{T'})$ should be not reduced alternating. Therefore, the diagram of $T'$ is either $(d)$ or $(e)$ of Figure~\ref{p24}.\\

If $(d)$ is the diagram of $T'$ then $0_4(\overline{T'})$ should be reduced alternating since $0_4(\overline{T})$ does. However, it is not reduced alternating.\\

By Lemma~\ref{T21},  $(e)$ is also not the diagram of $T'$.\\

The cases that a crossing of $T'$ by $\sigma_5^{\pm 1}$ or $\sigma_6^{\pm 1}$ is the only one closest crossing to $S^2$ are analogous to the previous two cases. \\

Now, assume that  crossings of $T'$ by $\sigma_2^{\pm 1}$ and $\sigma_4^{\pm 1}$ are the  closest crossings to $S^2$.\\

Then, take $0_3$-closure of $T'$ as the diagram $(f)$ of Figure~\ref{p24}. Since $0_3(\overline{T})$ is not reduced alternating, $0_3(\overline{T'})$ is also not reduced alternating. This implies that at least one of the dotted arcs  in the diagram $(f)$ of Figure~\ref{p24} is a real arc. (If a crossing by $\sigma_6^{\pm 1}$ is also a closest crossing to $S^2$ then clearly, $0_3(\overline{T'})$ is reduced alternating.) Both cases can be eliminated by a similar argument as in the previous cases.\\

Assume that  crossings of $T'$ by $\sigma_2^{\pm 1}$ and $\sigma_5^{\pm 1}$ are the  closest crossings to $S^2$.\\

Then, take $0_2$-closure of $T'$ as the diagram $(g)$ of Figure~\ref{p24}.  Since $0_5(\overline{T})$ is not reduced alternating, $0_5(\overline{T'})$ is also not reduced alternating. This implies that at least one of the dotted arcs  in the diagram $(g)$ of Figure~\ref{p24} is a real arc. Both cases can be eliminated by the same argument as in the previous case.\\

The cases having closest crossings by $\sigma_2^{\pm 1}$ and $\sigma_6^{\pm 1}$ or $\sigma_6^{\pm 1}$ and $\sigma_6^{\pm 1}$ are analogous to one of the previous cases.\\

Therefore, there is no counterexample when the diagram of $T$ is the diagram $(1)$ of Figure~\ref{p12}.\\
\begin{figure}[htb]
\includegraphics[scale=.55]{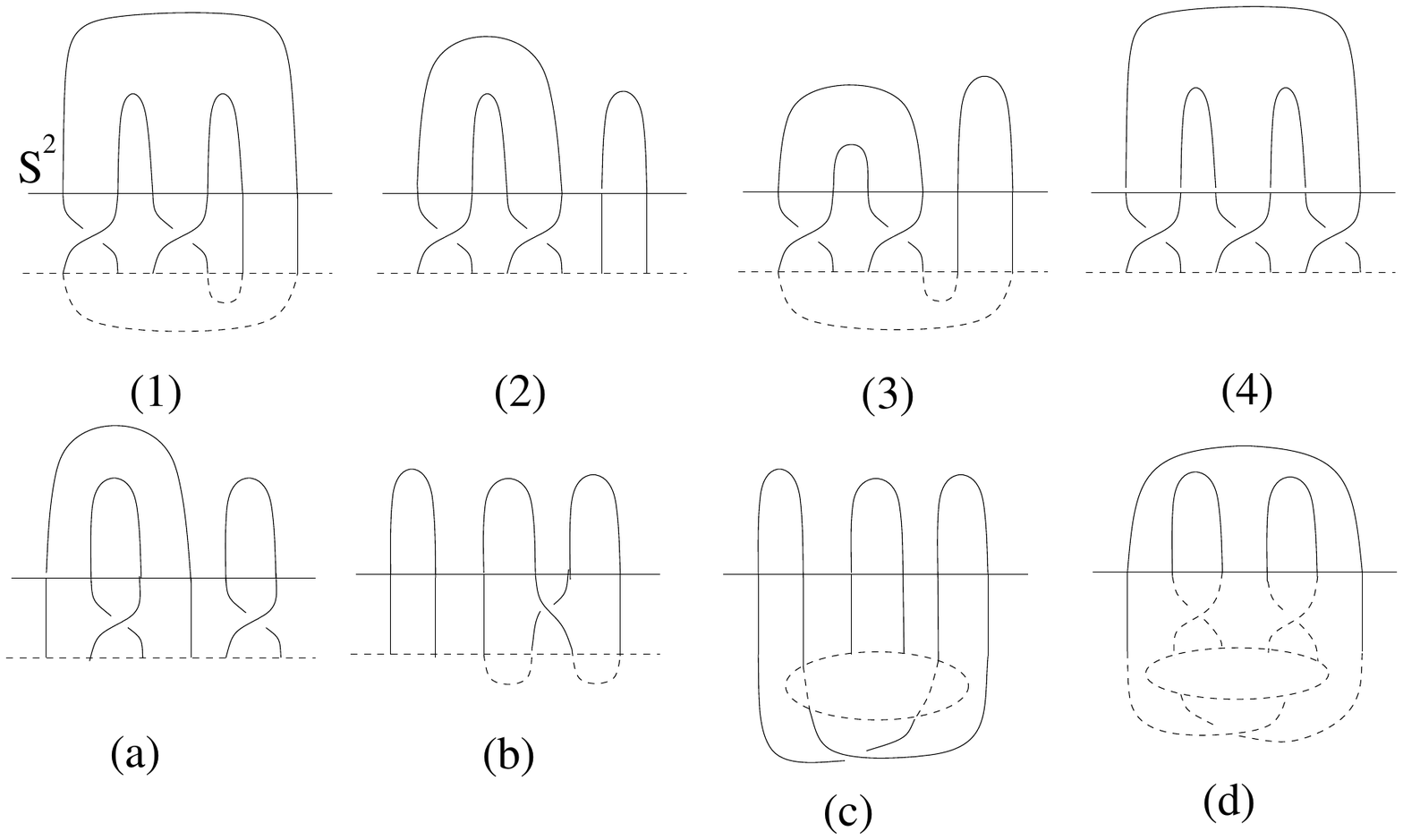}
\vskip -280pt
\caption{}
\label{p19}
\end{figure}

Case 2: Now, assume that two crossings which are obtained by $\sigma_1^{\pm 1}$ and $\sigma_3^{\pm 1}$ are closest crossings to $S^2$ and no other crossing can be a closest crossing to $S^2$.\\

Then $0_2$-closure of $T$ are reduced alternating link diagrams  as the diagrams $(2)$
of Figure~\ref{p19} since $T$ is reduced alternating.\\

We note that $T'$ cannot have a closest crossing by  $\sigma_1^{\pm 1}$ or $\sigma_3^{\pm 1}$.
Otherwise, it violates the rule that $|T|=|T'|$ is minimal. So, a closest crossing of $T'$ is by either $\sigma_2^{\pm 1}, \sigma_4^{\pm 1}, \sigma_5^{\pm 1}$ or $\sigma_6^{\pm 1}$.\\

If we have a closest crossing by $\sigma_2^{\pm 1}$ or $\sigma_5^{\pm 1}$, then we see that $0_2(\overline{T'})$ is not reduced alternating as the diagram $(a)$ of Figure~\ref{p19}. This makes a contradiction. \\

Now consider the case that a crossing by $\sigma_4^{\pm 1}$ or $\sigma_6^{\pm 1}$ is only two possible closest crossing to $S^2$ in $T'$. Then we also note that $0_1(\overline{T'})$ is not reduced alternating. (Refer to $(b)$ of Figure~\ref{p19}.) So, $0_1(\overline{T})$ is also not reduced alternating. This implies that the diagram $(1)$ of Figure~\ref{p19} is the diagram of $T$. i.e., at least one of the dotted arcs is a real arc.  Then $0_2(\overline{T})$ is reduced alternating as the diagram $(3)$ of Figure~\ref{p19}.\\

Assume that the crossing by $\sigma_4^{\pm 1}$ is the only one closest crossing to $S^2$ in $T'$. Since $0_3(\overline{T})$ is not reduced alternating, $0_3(\overline{T'})$ should be not reduced alternating. So, at least one of the dotted arcs of the diagram $(b)$ is a real arc. However, the right dotted arc cannot be real since $0_2(\overline{T'})$ should be reduced alternating. Also, the self crossing by the left dotted arc  impossible since the crossing by $\sigma_4^{\pm}$ is not also a crossing by $\sigma_3^{\mp 1}$ and it make a contradiction for $|T|=|T'|$ is minimal. \\

Now, assume that the crossing by $\sigma_6^{\pm 1}$ is the only one closest crossing to $S^2$ in $T'$. Then we can eliminate this case by a similar argument as in the previous case. (Refer to the diagram $(c)$ of Figure~\ref{p19}.)\\

Suppose that both $\sigma_4^{\pm 1}$ and $\sigma_6^{\pm 1}$ make a closest crossing to $S^2$ for $T'$. Then we can take $0_3$-closure of $T'$ and we conclude that $0_3(\overline{T'})$ should be reduced alternating by a similar argument as in the previous cases. However, $0_3(\overline{T})$ is not reduced alternating. Therefore, the second case also cannot be realized.\\

Case 3: Assume that three crossings which are obtained by $\sigma_1^{\pm 1},~\sigma_3^{\pm 1}$ and $\sigma_5^{\pm 1}$ are closest crossings to $S^2$ as the diagram $(4)$ in Figure~\ref{p19}.\\

 We note that a closest crossing to $S^2$ of $T'$ is by either $\sigma_2^{\pm 1}, \sigma_4^{\pm 1}$ or $\sigma_6^{\pm 1}$ since $|T|=|T'|$ is minimal.\\

We note that $0_1$-closure of $T$ is reduced alternating link diagram. However,  $0_1(\overline{T'})$ is not a reduced alternating as the diagram $(d)$ of Figure~\ref{p19}. So, this case is also impossible.\\

\begin{figure}[htb]
\includegraphics[scale=.65]{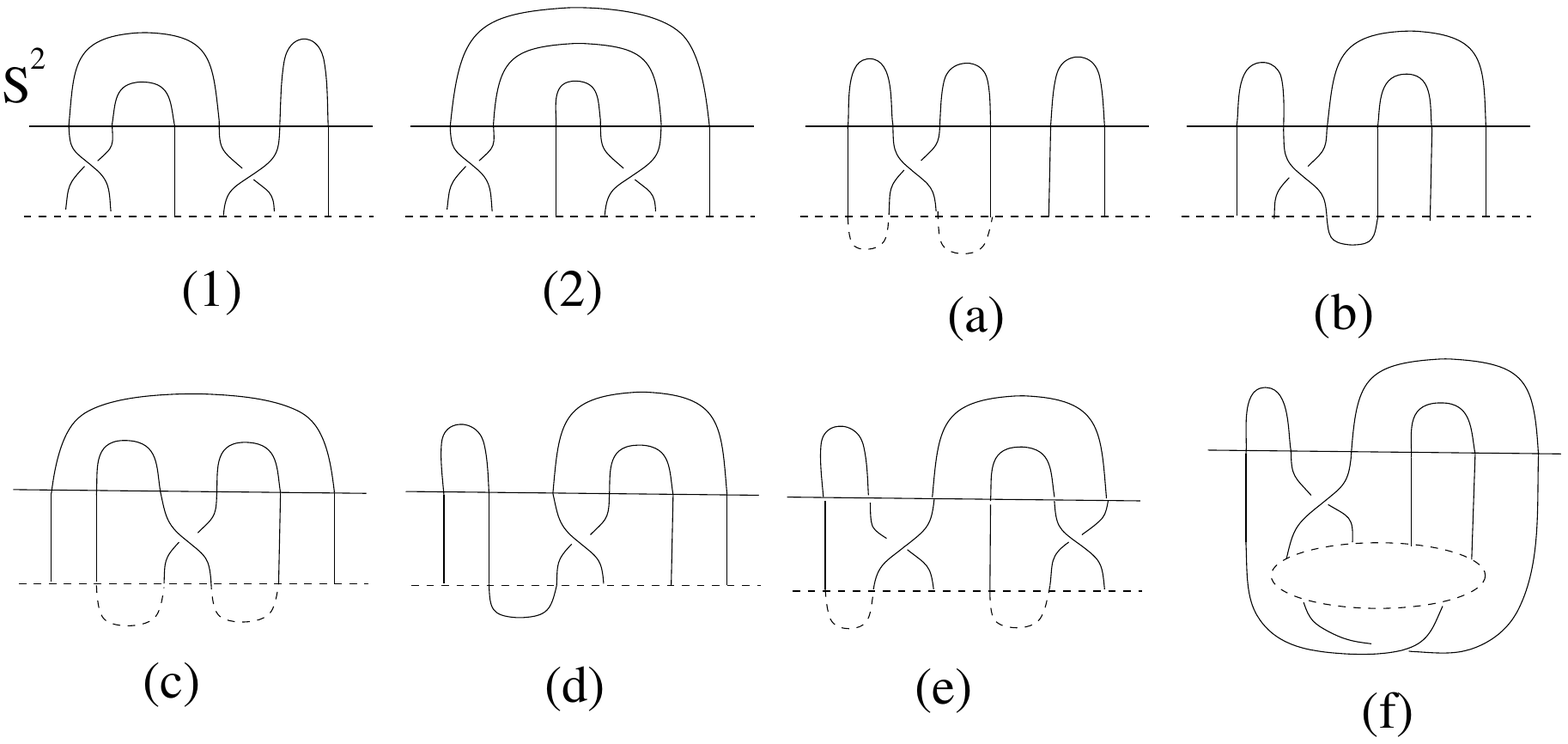}
\vskip -380pt
\caption{}
\label{p25}
\end{figure}
Case 4:  Assume that the two crossings which are obtained by $\sigma_1^{\pm 1}$ and $\sigma_4^{\mp 1}$  are closest crossings to $S^2$ as the diagrams $(1)$ and $(2)$ in Figure~\ref{p25}.\\

 We note that at least one of the diagrams $(1)$ and $(2)$ is reduced alternating by Lemma~\ref{T8}.
Also, we note that a closest crossing to $S^2$ of $T'$ is by either $\sigma_2^{\mp 1}, \sigma_3^{\pm 1}, \sigma_5^{\pm 1}$ or $\sigma_6^{\mp 1}$.\\

First, assume that the crossing by $\sigma_2^{\mp 1}$ is the only one closest crossing to $S^2$ of $T'$.\\

 Then take $0_3$-closure of $T'$ as the diagram $(a)$ of Figure~\ref{p25}. We note that one of the dotted arc should be real since $0_3(\overline{T'})$ is not reduced alternating  ($0_3(\overline{T})$ is not reduced alternating), $T'$ is reduced alternating and $|T'|\geq 2$. It is impossible to have the left dotted arc since we also can get the  crossing (by $\sigma_2^{\mp 1}$) from $\sigma_1^{\pm 1}$ in this case. It contradicts the assumption that $|T|=|T'|$ is minimal. Now, assume that the right dotted arc only can be realized.
 Then we note that $0_5(\overline{T'})$ is reduced alternating but $0_5(\overline{T})$ is not.
 This makes another contradiction.\\

Second, assume that the crossing by $\sigma_3^{\pm 1}$ is the only one closest crossing to $S^2$ of $T'$.  Consider the diagram $(c)$ and $(d)$. Then, we can make a contradiction by using a similar argument as in the previous case.\\

Now, assume that the crossings by $\sigma_2^{\pm 1}$ and $\sigma_5^{\pm 1}$ is the only two closest crossing to $S^2$. We note that $0_5(\overline{T'})$ is not reduced alternating since $0_5(\overline{T})$ is not reduced alternating. So, at least one of the dotted arcs in the diagram $(e)$ should be real arc. However, any of the arcs cannot be realized since the crossings by $\sigma_2^{\pm 1}$ or $\sigma_5^{\pm 1}$  also can be obtained from either $\sigma_1^{\pm 1}$ or $\sigma_4^{\mp 1}$ in these two cases. \\

Assume that the crossings by $\sigma_2^{\pm 1}$ and $\sigma_6^{\pm 1}$ is the only two closest crossing to $S^2$ of $T'$. Then $0_5(\overline{T'})$ is reduced alternating as the diagram $(f)$ of Figure~\ref{p25}. However, $0_5(\overline{T'})$ is not reduced alternating. This makes a contradiction.\\

We note that the rest of cases also can be eliminated by a similar argument in the previous cases since we can modified each of the remaining cases into a case I already covered.\\

Therefore, it is impossible to have a closest crossing by $\sigma_1^{\pm 1}$ to satisfy the assumption that $v_T=(-a^{-3})^kv_{T'}$ for some $k$ but $\mathbb{T}\not\approx \mathbb{T}'$.\\

The cases having a closest crossing to $S^2$ of $T$ by $\sigma_j^{\pm 1}$ for $2\leq j\leq 6$ are also eliminated by a similar argument as in the previous cases by considering  $n\times 60^\circ$ counterclockwise rotations ($1\leq n\leq 5$). (Refer to Figure~\ref{p10}.)\\

This completes the proof.

\end{proof}

\begin{Con}\label{T14} Suppose that $T$ and $T'$ are two rational $3$-tangles.
Then $T\approx T'$ if and only if $v_T=(-a^{-3})^kv_{T'}$ for some integer $k$.
\end{Con}

\begin{figure}[htb]
\includegraphics[scale=.34]{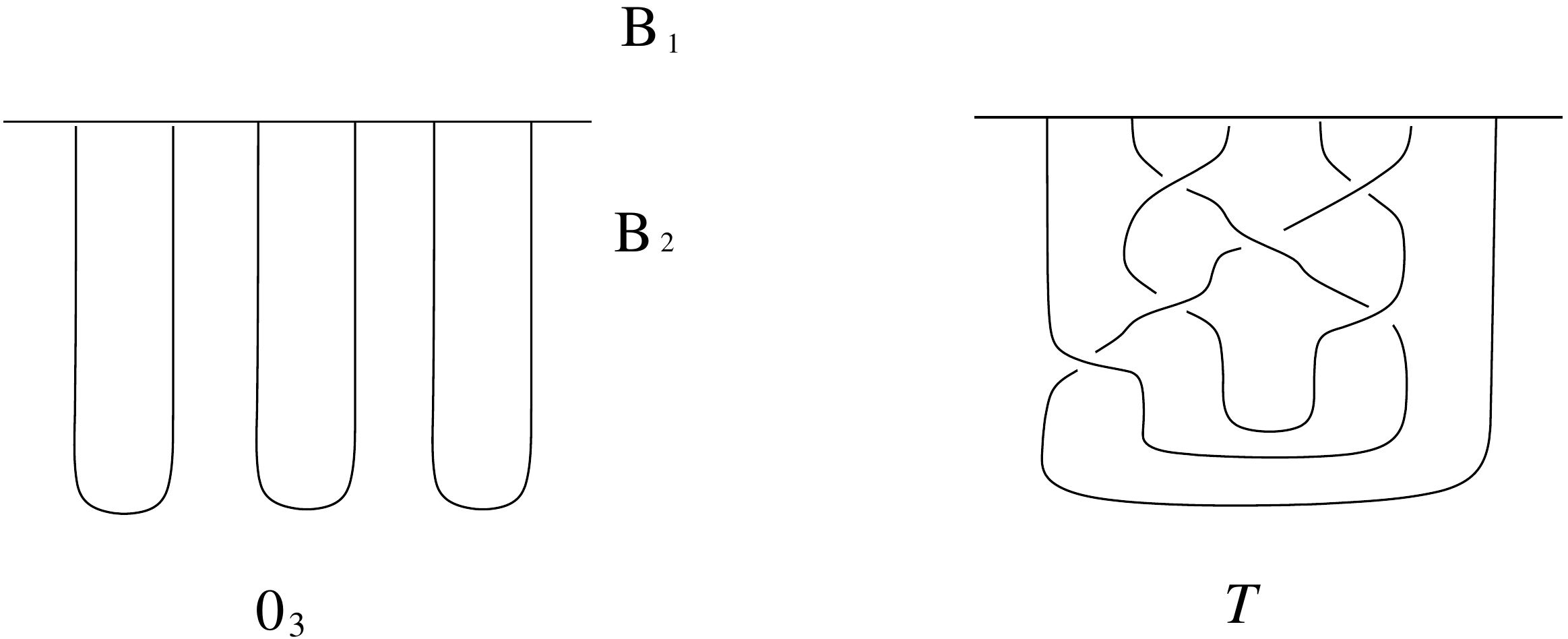}
\caption{}
\label{p16}
\end{figure}
$\bf{Example}:$ Figure~\ref{p16} gives us  an example of two rational alternating $3$-tangles that are distinguished by the invariant. We note that $0_3(\overline{T})$ is the Borromean rings. So, if you take any two of the three strings in $T$ then we get a trivial rational 2-tangle.\\

First of all, we see that $v_{0_3}=(0,0,1,0,0)$.\\

I found a method to calculate the Kauffman bracket of 6-plat presentations of links by using a presentation of braid group $\mathbb{B}_6$ into a group of $5\times 5$ matrices. (Refer to~\cite{9}.)\\

By the method, we have $v_{T}=(-a^{-6}+3a^{-2}-3a^2+a^6,-2a^4+a^8+1,-2a^6+a^{10},-a^4,-2a^4+a^8+1)$.\\

Therefore, $v_{0_3}\neq (-a^{-3})^kv_T$ for any $k$.\\

This implies that $0_3\not\approx T$.

\end{document}